\documentclass{cimart}

\usepackage{amssymb}

\usepackage{mathtools}
\newcommand{\CC}{\mathbb{C}}
\newcommand{\NN}{\mathbb{N}}
\newcommand{\QQ}{\mathbb{Q}}
\newcommand{\RR}{\mathbb{R}}

\title[Elementary fractal geometry]{Elementary fractal geometry.\\ 4. Automata-generated topological spaces}

\author{Christoph Bandt}

\abstract{
	Finite automata were used to determine multiple addresses in number systems and to find topological properties of self-affine tiles and finite type fractals. We join these two lines of research by axiomatically defining automata that generate topological spaces. Simple examples show the potential of the concept.  Spaces generated by automata are topologically self-similar.  Two basic algorithms are outlined. The first one determines automata for all $k$-tuples of equivalent addresses from the automaton for double addresses. The second one constructs finite topological spaces which approximate the generated space. Finally, we discuss the realization of automata-generated spaces as self-similar sets.
	}



\authorinfo{Institute of Mathematics, University of Greifswald,  Germany}{bandt@uni-greifswald.de}

\VOLUME{33}
\YEAR{2025}
\ISSUE{2}
\NUMBER{4}
\DOI{ https://doi.org/10.46298/cm.12647}

\begin{document}

\section{Overview}\label{intro}

The basic principle of number systems is to assign addresses to points. We have an address map $\varphi :S\to X$ from a symbolic space $S$ to a set $X$ of points. The simplest symbolic space is the space of sequences  $s=i_1i_2\dots$ from the digit set $D=\{ 0,1,\dots,m-1\}$. Two symbol sequences $s,t$ are called equivalent if they address the same point: $\varphi(s)=\varphi(t)$.  This paper studies algorithms that describe multiple addresses. If the set of pairs $(s,t)$ of equivalent addresses is generated by a finite automaton, the address map $\varphi$ is called automatic, and the topological quotient space $X=S/\varphi$ is automata-generated. 

Gilbert~\cite{Gi82,Gi87} started work in this direction in 1982. He constructed automata which determine double and triple expansions in number systems with complex bases $-n+i$, for $n=1,2,\dots $ Thurston emphasized the connection between number systems, self-similar tilings and automata in his influential lecture notes~\cite{T89}. This work was extended in numerous papers on canonical number systems (cf.~\cite{akithusurvey} and~\cite[Section 2.4]{FS09}) and on the topology of self-affine tiles in the plane~\cite{ALT, DN05, LL13, LL07, Lo, LZ17, ScT} and in space~\cite{TZ20}. Roughly speaking, self-affine tiles are the unit intervals of canonical number systems. Their vertices, which are interesting from a~geometrical viewpoint, have three or more addresses.

Symbolic dynamics is a~related field where addressing plays a~central part. Milnor and Thurston~\cite{MiTh} characterized unimodal maps on the interval by their kneading sequence which represents the double address of the critical point. There are special parameters where this address is preperiodic. For complex rational functions, Thurston discussed the symbolic description of Julia sets in the post-critically finite (p.c.f.) case~\cite{T85}. They were considered topologically self-similar sets by Kameyama~\cite{Ka93}. In fractal geometry, p.c.f.\ sets have been frequently studied because of their simple structure. Brownian motion and differential equations can be defined on such spaces~\cite{Kig, KongLau, Str, tep}. All p.c.f.\ fractals are generated by automata, as shown in~\cite{RW23} and in Section~\ref{exam}.

Both p.c.f. fractals and self-affine tilings belong to the class of finite type self-similar sets $X$ which can be described by an automaton, termed neighbor graph in~\cite{BM09}. Its original purpose was to check a~separation condition~\cite{Ba00}. It turned out that the automaton provides much information on the topology of $X$, as well as the Hausdorff dimension of the boundary and local dimensions~\cite{HR22}. Implemented in the IFStile package of Mekhontsev~\cite{M,Mth}, millions of fractals could be screened and interesting examples selected~\cite{EFG1,EFG3}.

However, self-similar sets generated by similitudes are mainly studied in the plane. In three dimensions, similitudes seem to be too special. Several authors used a~combinatorial approach to study fractal topology, putting aside metric features like similarity and Haus\-dorff dimension~\cite{baldwin,BK,Hata,Ka93,Pth,tep}. Zhu and Rao~\cite{ZR16} defined topology automata for fractal squares. We follow this line of research.

The new point is that we do not derive the automaton from a~number system, tile or fractal. We take the automaton as a starting point: as a~tool to construct a~topological space. In Section~\ref{basic} we axiomatically introduce the class of automata which accepts the double addresses in an address map. We give many simple examples and illustrations in Sections~\ref{auto2} and~\ref{exam}. Graphs were drawn by MATLAB, fractals by IFStile~\cite{M}. In Section~\ref{topsel} we prove that automata-generated spaces are topologically self-similar, even in a~more general setting.

Two basic algorithms are described. In Section~\ref{comp} we derive automata for the triple and $k$-tuple addresses from the original automaton for double addresses. For an automaton with three states and three digits, we find all points in $X$ with 4, 6, and 12 addresses.
In Section~\ref{inver} we construct the topological space $X$ from finite space approximations, unfolding the edge structure of the automaton. Since the automaton is finite, topological properties of the limit should be computable on a~certain finite level.
This is a work in progress. In the final Section~\ref{IFSex} we discuss the realization of an automata-generated space as a~self-similar set in the complex plane. Many open questions appear in the last part of the paper. A long-term target is to establish databases of recursively defined topological objects, maintained and developed by computer.

The automata discussed in this paper are quite similar to those studied by Frougny and Sakarovitch~\cite{FS09}, but they are used for a more geometric purpose.

\section{The concept of a~topology-generating automaton}\label{basic}
Throughout, $D=\{ 0,1,\dots,m-1\}$ is our alphabet, the set of digits for numeration, denoted by $i,j,i_k,u_k,s_k$. Words $i_1\dots i_n\in D^n$ will be denoted by $u,v,w$ and sequences $i_1i_2\dots$ by $s,t$, or $\overline{u}=uuu\dots$ for periodic sequences. The set of words is $D^* =\bigcup_{n=0}^\infty D^n $.
The set of sequences is our symbolic space $S$, the full one-sided shift over $D$ with the usual product topology. Elements of $S$ are called addresses.
\[
S=D^\NN =\{ s=s_1s_2\dots\, |\, s_k\in D \} .
\]
An address map is a~function $\varphi :S\to X$ from $S$ to any set $X$. This map induces the quotient topology on $X$ which transforms $X$ into a~compact space. When we start from a~compact topological space $X$, we require that $\varphi$ is the quotient map. Two addresses $s,t$ are said to be equivalent if $\varphi(s)=\varphi(t)$. As usual, we identify the equivalence relation with the set
\[
L_\varphi =\{ (s,t)\, |\, s,t\in S, \ \varphi(s)=\varphi(t)\} \subseteq S\times S.
\]
We call the address map $\varphi$ automatic and the topology of $X$ automata-generated if there exists a~finite automaton $G$ which accepts $L_\varphi $. In this paper, the automaton is the starting point, and the equivalence relation and the topology of $X$ will be derived from it. We begin with a~basic example.

\begin{figure}[h!t]
\begin{center}
\includegraphics[width=0.6\textwidth]{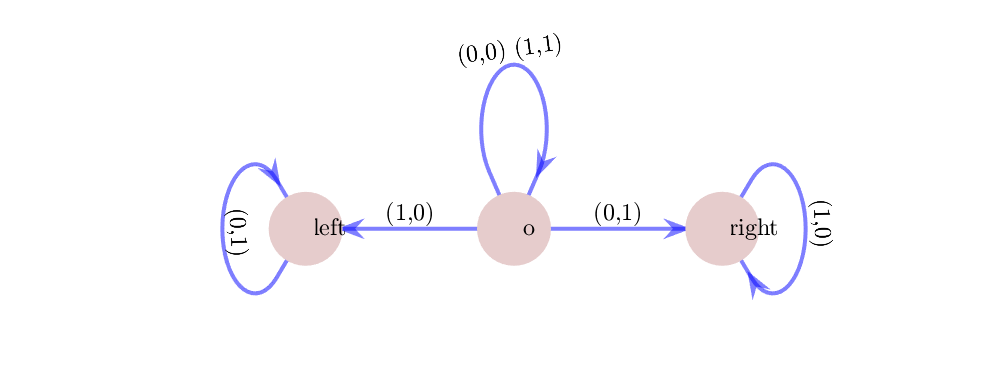}
\end{center}
\caption{The automaton for double addresses of binary numbers.}\label{fig1}
\end{figure}

\begin{example}[Binary numbers]
\label{ex1}
Let $D=\{ 0,1\}, X=[0,1]$ and $\varphi(i_1i_2\dots)=\sum_{k=1}^\infty i_k\cdot 2^{-k}$. Then $\varphi(s)=\varphi(t)$ if and only if either $s=t$ or there is an integer $n\ge 0$ and $i_1,\dots,i_n\in D$ such that $s=i_1\dots i_n 0\overline{1}$ and $t=i_1\dots i_n1\overline{0} $, or conversely. The corresponding automaton $G$ is drawn in Figure~\ref{fig1}. The initial state is $o$, and the input alphabet is the set $D^2$ of pairs of symbols. A pair of words or sequences $(u,v)$ is accepted if there is a~directed path of edges with labels $(u_1,v_1), (u_2,v_2),\dots$ in $G$ which starts in $o$. For $u=v$ this path will consist of loops from $o$ to itself. When $u$ is lexicographically smaller than $v$, the edge with label $(0,1)$ will lead to the state \textit{right}, and all following labels must be $(1,0)$. In this way, the automaton defines the equivalence of binary addresses and thus the topology of the interval.
\end{example}
Before we formally define our automata, we note how the equivalence of addresses is expressed on the level of words of length $n$.
For a~word $u=u_1\dots u_n$ the corresponding cylinder set in the symbolic space is defined as
\[
S_u=\{ u_1u_2\dots u_ns_{n+1}s_{n+2}\dots\, |\, s_k\in D \} \subseteq S.
\]

\begin{proposition}\label{prop1}
For any map $\varphi:S\to X$ and any $s,t\in S$ we have
\[
\varphi(s)=\varphi(t) \mbox{ if and only if }\ \varphi(S_{s_1\dots s_n})\cap \varphi(S_{t_1\dots t_n}) \not=\emptyset \ \mbox{ for }\ n=1,2,\dots
\]
\end{proposition}

 \begin{proof} If $x=\varphi(s)=\varphi(t)$ then $x$ belongs to all intersections on the right-hand side. On the other hand, $\bigcap_n S_{t_1\dots t_n}=\{ t\} $, and the $M_n=\varphi(S_{s_1\dots s_n})\cap \varphi(S_{t_1\dots t_n})$ form a~decreasing sequence of closed sets. If they are all nonempty, $M=\bigcap_n M_n$ is nonempty by compactness, and must coincide with $\{\varphi(t)\}$ as well as with $\{\varphi(s)\} $.
  \hfill \end{proof}

The sets $X_{t_1\dots t_n}=\varphi(S_{t_1\dots t_n})$ will be called the pieces of level $n$ for the address map $\varphi $. The proposition says that two sequences address the same point if and only if the corresponding pieces of level $n=1,2,\dots$ intersect each other. In example~\ref{ex1}, pieces of level $n$ are binary intervals $[x,x+2^{-n}]$ with $x=\varphi(t_1\dots t_n\overline{0})$. For $n=1$ we have $X_0=[0,\frac12]$ and $X_1=[\frac12,1]$ with one common endpoint. If we start with the edge with label $(i,j)=(0,1)$ then $X_j$ is on the right of $X_i$ which explains the name of the state in Figure~\ref{fig1}. For $n=2$, only $X_{i1}$ and $X_{j0}$ intersect, where the latter remains on the right. In this way, the automaton is constructed.
A state can be interpreted as a~relative position of the two pieces.

\begin{definition}[Topology-generating automaton]
\label{def1} A topology-generating automaton, denoted by $G=(V,D^2,E,o)$, consists of a~finite directed graph $(V,E)$ where vertices represent states and edges represent transitions labelled by the input alphabet $D^2 $, and $o$ denotes the initial state. The graph can contain loops and multiple edges which can be drawn as single edges with multiple labels. The following properties are required.
\begin{enumerate}
\item Each vertex $c$ has an outgoing edge and can be reached by a~directed path from $o$.
\item A vertex $c$ must not admit two outgoing edges with the same label $(i,j)$.
\item To each state $c$ there is an `inverse' state $c-$ such that for every edge between vertices $b$ and $c$ labelled with $(i,j)$, there is an edge labelled with $(j,i)$ between $b-$ and ${c-}$. The conditions $o-=o$ and $(c-)-=c$ for $c\in V$ are fulfilled.
\item The initial state $o$ has loops with label $(i,i)$ for each $i\in D$.
\end{enumerate}

We use the convention that an input $(i,j)$ is accepted at state $c\in V$ if there is an edge with label $(i,j)$ starting in $c$.
A pair of words $(i_1\dots i_n, j_1\dots j_n)$ or sequences $(i_1i_2\dots, j_1j_2\dots)$ is accepted if there is a~directed path of edges starting in $o$ with labels $(i_1,j_1),(i_2,j_2),\dots$. The language of accepted pairs of words and the set of accepted pairs of sequences are denoted by $L(G)\subseteq (D^2)^*$ and $L^\infty(G)\subseteq S^2 $, respectively.
\end{definition}

We used common terminology as in~\cite{BeRi,Epstein}. While automata for number systems as treated by Frougny and Sakarovitch~\cite{FS09} use the input alphabet $D$, we need $D^2 =D\times D$ to process pairs of symbol sequences. The absence of a~set of final states was explained by Thurston as follows: ``The language $L(G)$ is prefix-closed if every prefix of a~word in $L(G)$ is also in $L(G)$, or in other words, if every non-accept state has arrows only to other non-accept states. In such a~case, we may collapse all non-accept states in a~single fail state, with all arrows leading back to itself. It is convenient to omit the fail state and all roads leading to it. Whenever an input gives you directions where there is no corresponding arrow, you immediately fail with no chance for reinstatement.''~\cite[p. 31]{T89}.

According to Proposition~\ref{prop1} we want to describe the relation $\varphi(S_u)\cap \varphi(S_v) \not=\emptyset$ for words $u,v$ on $D^2 $. This relation is prefix-closed since $S_{u'}\supseteq S_u$ for any prefix $u'$ of $u$. So Thurston's convention applies here.
Moreover, if the relation is fulfilled for $(u,v)$ then it must also be fulfilled by $(ui,vj)$ for some $(i,j)\in D^2 $, since $S_u=\bigcup_{i\in D} S_{ui}$. This explains the outgoing edges in property 1. These edges guarantee that each directed path of edges can be extended indefinitely, in a~finite graph through directed cycles. Each vertex should have an incoming edge and should be reached from $o$ because otherwise it is obsolete.

Property 2 is quite natural. It says that an accepted pair of addresses $(s,t)$ determines a~unique directed path in the graph. Property 3 expresses the symmetry of the relation $\varphi(S_u)\cap \varphi(S_v) \not=\emptyset $. If $(u,v)$ is accepted at a~state $c$ then $(v,u)$ must also be accepted at a~state which we call $\pi(c)$. Then if $(ui,vj)$ is also accepted, the same is true for $(vj,ui)$. As discussed in Section~\ref{topsel}, property 4 can be replaced by a~weaker condition. It is needed to accept pairs of equal sequences $(s,s)$, and is equivalent to a~self-similarity condition.

Given an automaton with these properties, we formally define the quotient space $X$ as the set of all equivalence classes $x_s$ of addresses $s$. The address map $\varphi$ will assign to each address $s\in S$ its equivalence class.
\[
\varphi (s)= \{ t\in S \, |\, (s,t) \mbox{ is accepted by }G\} = x_s.
\]
One way to define the quotient topology on $X$ is to say that a~sequence $(x_{s^n})_{n=1,2,\dots}$ with $s^n \in S$ converges to $x_s$ in $X$ if and only if $s^n$ converges to $s$ in $S$. However, there can be $t$ with $x_s=x_t$ and $t^n$ with $x_{s^n}=x_{t^n}$. We show that the convergence does not depend on the choice of representatives.

\begin{proposition}
 
The following holds for any topological automaton $G$.
\begin{enumerate}
\item[ (i)] Let $(s^n)_{n=1,2,\dots}$ and $(t^n)_{n=1,2,\dots}$ be convergent sequences of addresses with limit sequences $s$ and $t$ in $S$. If $(s^n,t^n)$ is accepted by $G$ for $n=1,2,\dots $. then $(s,t)$ is also accepted by $G$.
\item[(ii)] The quotient space $X$ generated by $G$ is a~compact Hausdorff space. It is metrizable.
\end{enumerate}
\label{prop2}
\end{proposition}

 \begin{proof} 
 \begin{enumerate}

    \item [(i)] In $S$ we have coordinatewise convergence. Thus the first $k$ digits of $s^n$ must agree with the digits $s_1s_2\cdots s_k$ of the limit sequence $s$ for all $n$ greater than some $n(k)$. Similarly, $t^n$ starts with $t_1t_2\cdots t_k$ for large enough $n$. Since $(s^n, t^n)$ is accepted by $G$ for all $n$, this implies $\varphi(S_{s_1\cdots s_k})\cap \varphi(S_{t_1\cdots t_k}) \not=\emptyset$. This holds for all $k$, so $(s,t)$ is accepted by $G$.
\item [(ii)] By definition of the quotient topology, $\varphi$ is continuous. The space $S$ is compact, and a~continuous image of a~compact space is always compact. A space is Hausdorff if limits are uniquely determined. This is what we proved in (i) for the quotient topology of $X$. Since $S$ has a~countable base, this holds for $X$ and implies that $X$ is metrizable.
 \end{enumerate}
\end{proof}

In section~\ref{inver} we shall give a~more concrete construction of the space $X$.
First we present small examples of topological automata. \textit{In the whole paper we confine ourselves to equivalence relations with finite equivalence classes, with size smaller than a~constant $C$.}
Thus we shall exclude identifications of periodic addresses:

\begin{proposition}
  Let all equivalence classes of addresses be finite, and let $(s,t)$ be a~pair of different sequences accepted by the automaton $G$. Then $s$ is not a~periodic sequence and cannot be written as $s=ut$ for a~word $u\in D^* $.
\label{prop2b}
\end{proposition}

 \begin{proof} If a~periodic sequence $s=\overline{w}$ is identified with $t$, then $s=ws$ is identified with $wt$ by property 4. By induction $s$ is also identified with $w^k t$ for $k=2,3,\dots\, $. Thus the resulting equivalence class is infinite.

If $t$ is identified with $ut$ then by property 4 $ut$ is identified with $u^2 t$ and $u^2 t$ with $u^3 t$ and so on. Again we have an infinity of equivalent addresses $u^k t, k=0,1,2,\dots$, and $\overline{u}$ will also belong to this class because the equivalence relation is closed according to Proposition~\ref{prop2}. So we have the same conclusion as above.
 \hfill \end{proof}

A particular case appears when there is a~path from $o$ to $o$ starting with labels $i_1\not= j_1$. If $u,v$ are the words addressing this path, then not only $\overline{u}$ and $\overline{v}$ are identified. Every sequence in $\{ u,v\}^\infty$ will belong to the equivalence class which therefore has the cardinality of the continuum. There are important number systems like $\beta$-numeration with a golden mean base which have this property. They are also interesting from the topological viewpoint but require methods that are not studied here. For that reason, we assume that the loops of property 4 are the only incoming edges to $o$.

\section{Automata with two states}\label{auto2}
Figure~\ref{fig1} showed a~topological automaton with two states - the initial state will not be counted. It can be easily modified to describe the decimal numbers with $D=\{ 0,1,\dots,9\}$: the edge from $o$ to $right$ will have labels $(0,1),(1,2),(2,3),\dots,(8,9)$, and the loop at state $right$ gets the label $(9,0)$. Similarly we can describe the number system with respect to any positive integer base.

Let us go back to $D=\{ 0,1\}$ and look for modifications of Figure~\ref{fig1}. Because of properties 3 and 4, the labels cannot be changed, except for one label of a~loop, say at $right$. If we replace $(1,0)$ by $(0,1)$ or $(0,0)$, however, the infinite path from $o$ to $right$ and then traversing the loop will determine the sequences $s,t$ with $s=\overline{0}$. This contradicts Proposition~\ref{prop2b}. If we take $(1,1)$ for that label, we would have $t=\overline{1}$. The conclusion is that for the graph of Figure~\ref{fig1} no other labeling defines a~topological automaton.

\begin{figure}[h!t]
\begin{center}
a \includegraphics[width=0.3\textwidth]{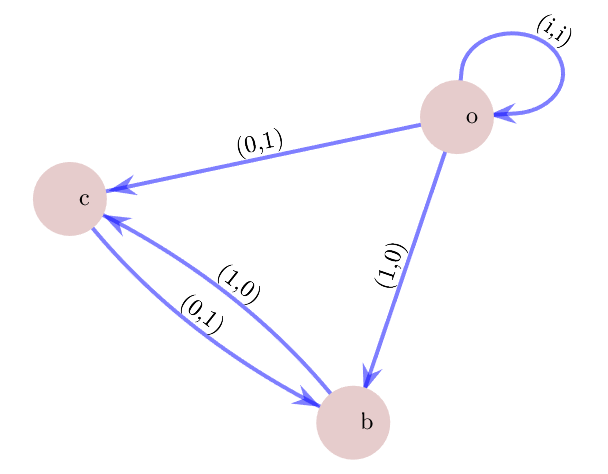} \
b \includegraphics[width=0.3\textwidth]{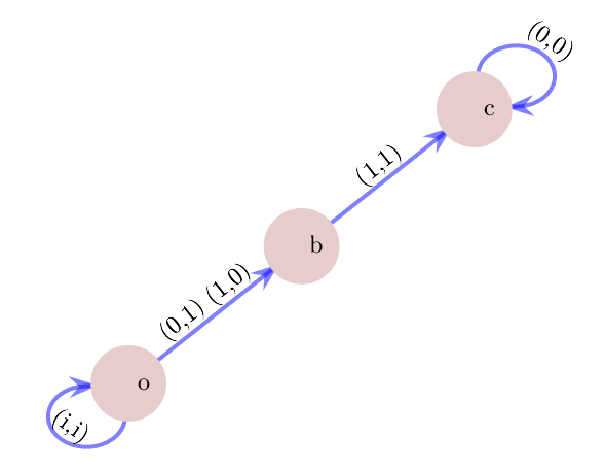} \
c \includegraphics[width=0.3\textwidth]{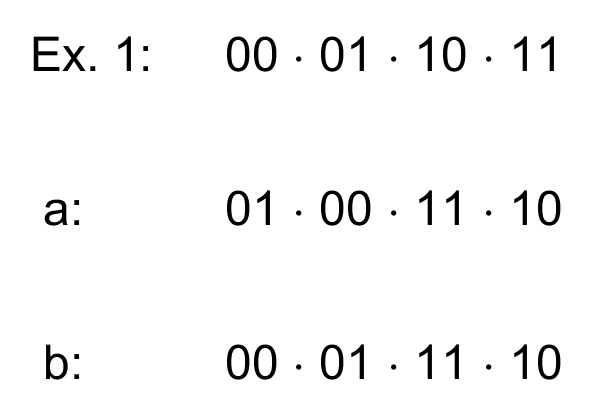}
\end{center}
\caption{Topological automata for two number systems. \textbf{a:} base -2, \textbf{b:} symbolic dynamics of the tent map. \textbf{c:} The order of second-level pieces in the interval $X$ for Figure~\ref{fig1},~a,~and~b.}\label{fig2}
\end{figure}

So we have to change the graph: we replace the loops by arrows to the inverse state, as shown in Figure~\ref{fig2}\textcolor{blue}{a}. The label $(i,j)$ for the arrow from $c$ to $b$ can be chosen, the other must be $(j,i)$. If we take $(0,0)$ or $(1,1)$, we shall again have sequences $s,t$ with $s=\overline{0}$ or $t=\overline{1}$. The label $(1,0)$ would give $s=\overline{01}$.
The only possible label for the edge $(c,b)$ is $(0,1)$. This automaton describes the number system with base $-2$, with address map $\varphi(i_1i_2\dots)=\sum_{k=1}^\infty i_k\cdot (-2)^{-k}$ and $X=[-\frac23, \frac13 ]$. The `identification formula' is $0\overline{01}\sim 1\overline{10}$ while it was $0\overline{1}\sim 1\overline{0}$ for binary numbers.

The third graph has two self-inverse states $b,c$. Without loss of generality there is a~double arrow from $o$ to $b$ which must have labels $(0,1)$ and $(1,0)$, because of properties 2-4. There must be an arrow from $b$ to $c$. Arrows between self-inverse states can only have labels of the form $(i,i)$. Let us take the label $(1,1)$ for the edge $(b,c)$. By property 1 there is an arrow starting in $c$. If it leads back to $b$, we obtain a~contradiction with periodic $s$. So this edge must be a~loop, with label $(0,0)$ to avoid periodicity. See Figure~\ref{fig2}\textcolor{blue}{b}. The identification formula is $01\overline{0}\sim 11\overline{0}$. This automaton represents the symbolic dynamics of the tent map $T:[0,1]\to [0,1], \ T(x)=2x$ for $x\le\frac12$ and $T(x)=2(1-x)$ for $x\ge\frac12 $. We have $\varphi^{-1}(x)=j_0j_1j_2\dots $ with $j_k=0$ if $T^k (x)\le\frac12$ and $j_k=1$ if $T^k (x)\ge\frac12 $. The critical point $x=\frac12$ has two addresses. The tent map is conjugate to the quadratic function $g(x)=x^2 -2$ on its Julia set $[-2,2]$ so $X$ could be called a~Julia set in this case.
An equivalent automaton is obtained by exchanging $(0,0)$ and $(1,1)$ for which we have to put the tent map upside down. We thus extended an old result of Hata~\cite[p. 399]{Hata} and Bandt and Keller~\cite[Proposition 4]{BK} to automata-generated spaces.

\begin{proposition}\label{prop3}
There are three different topological automata with two states (not counting $o$) and a~two-digit alphabet. They describe binary numbers, the number system with base $-2$ and the symbolic dynamics of the tent map. In all cases, the generated space is an interval. \hfill $\Box$
\end{proposition}

\begin{figure}[h!t]
\begin{center}
\includegraphics[width=0.5\textwidth]{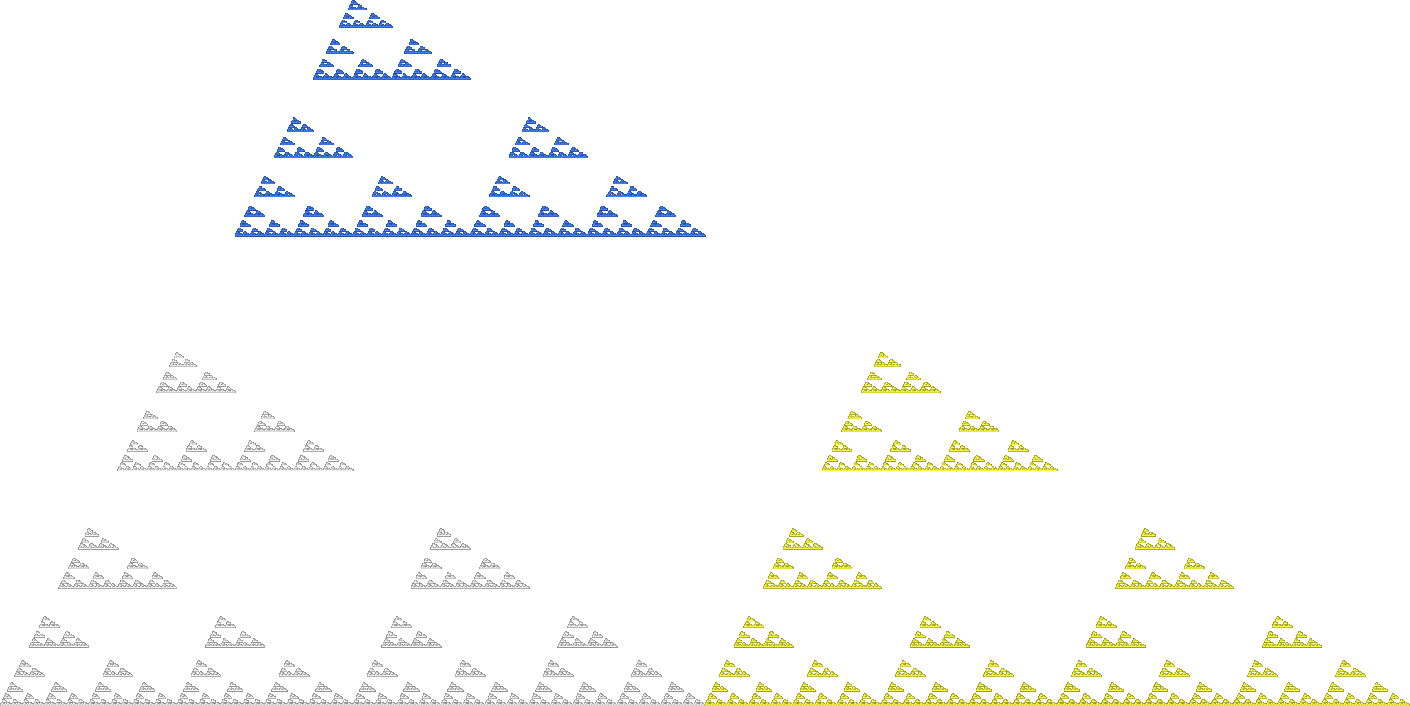}
\end{center}
\caption{Adding digit 2 to the automaton of binary numbers we generate a~disconnected space.}\label{fig3}
\end{figure}

Now let us take alphabets with more than two digits. When we add a~digit 2 in Figure~\ref{fig1} without adding any labels with 2, except for the mandatory loop $(2,2)$ at $o$, we obtain a~space $X$ where $X_2$ is disconnected from $X_0\cup X_1$, and $X_{02}$ is disconnected from $X_{00}\cup X_{01}$ etc. One metric realization is sketched in Figure~\ref{fig3}. The space consists of countably many intervals and a~continuum of isolated points since it has a Hausdorff dimension greater than 1. In this way, any connected space can be transformed into an archipelago by adding one more digit. For this reason, we focus our study on connected spaces. Disconnected spaces described by two-state automata and their Lipschitz equivalence were studied by Zhu and Yang~\cite{ZY18}.

When we use digit 2 and add the label $(0,2)$ on the edge $(o,right)$ (and $(2,0)$ on $(o,left)$ because of property 3) then we get a~dendroid, a~Hata tree, as indicated in Figure~\ref{fig4}. The formula for double addresses would be $0\overline{1}\sim 1\overline{0}, 0\overline{1}\sim 2\overline{0}$. However, the automaton in Figure~\ref{fig1} is not complete. It does not describe the equivalence $1\overline{0}\sim 2\overline{0}$ which must hold by transitivity. The complete automaton with three states is shown in Figure~\ref{fig4}. Since our automata accept pairs of addresses, incompleteness is possible for triple and multiple addresses. In the context of Figure~\ref{fig1} with added digits and edge labels, this will always happen when two labels at the same edge share a~digit, like $(i_1,j),(i_2,j)$.

\begin{figure}[h!t]
\begin{center}
\includegraphics[width=0.54\textwidth]{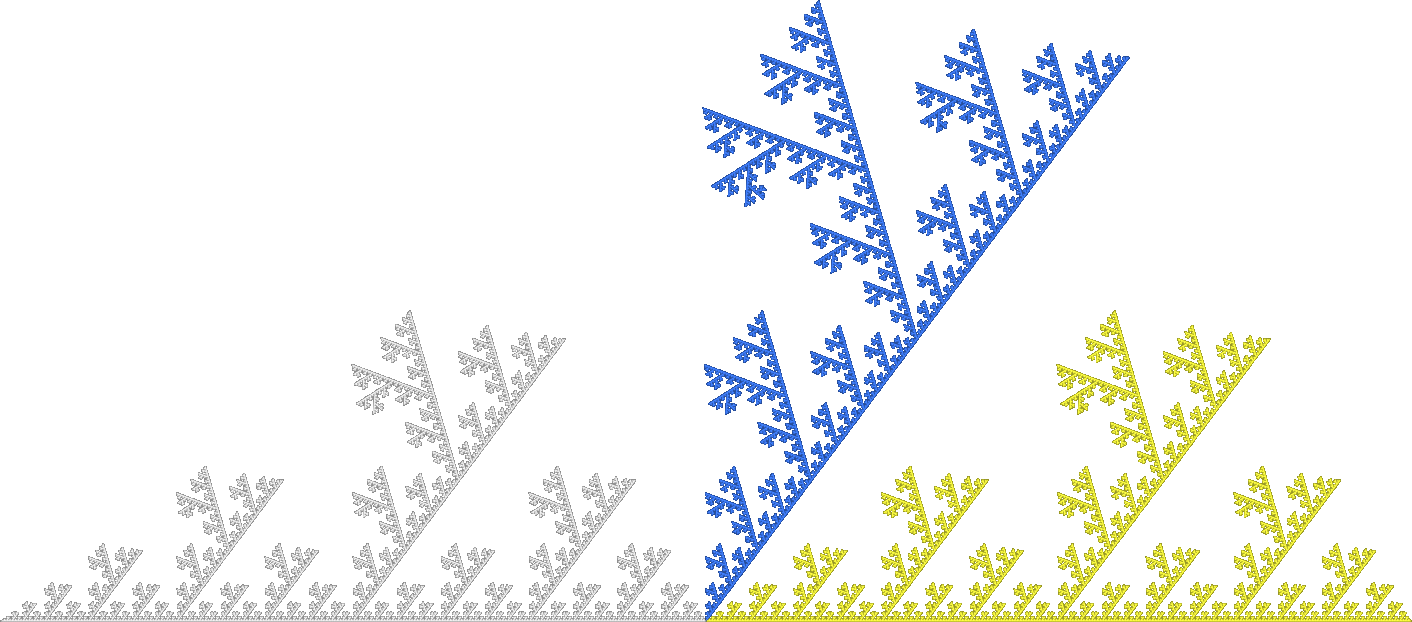} \
\includegraphics[width=0.44\textwidth]{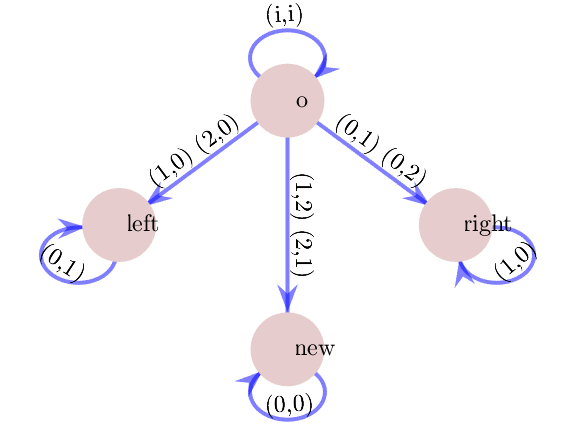}
\end{center}
\caption{Left: A Hata tree is generated from the automaton in Figure~\ref{fig1} by adding digit 2 and two edge labels. Right: The triple address at the branching point in the middle requires one more state for completeness. }\label{fig4}
\end{figure}

\begin{proposition}\label{prop4}
Let $G$ be a~topological automaton with $m$ digits and two states that are inverse to each other. Suppose that $G$ completely describes the equivalence of addresses and that $X$ is connected. Then $X$ is an interval, and $G$ defines the double addresses of a~number system, either with base $m$ or with $-m$.
\end{proposition}

 \begin{proof} We consider only the case of Figure~\ref{fig1}. The case of Figure~\ref{fig2}a is similar. Consider an undirected graph $H$ with vertex set $D$ and with edges ${i,j}$ for all labels $(i,j)$ or $(j,i)$ of the edge from $o$ to \textit{right.} A label $(i,i)$ is not possible since then the state \textit{right} would be self-inverse. Since $X$ is connected, the graph $H$ is connected and thus has at least $m-1$ edges. Since $G$ is complete, the first and second coordinates of the labels must be different, as noted above. That implies that $H$ has exactly $m-1$ edges which form an undirected path. By renaming the digits we can assume that the labels are $(0,1), (1,2),\dots,(m-2,m-1)$. Then the loop at state \textit{right} can have only label $(m-1,0)$ because of the Proposition~\ref{prop2b}. This gives the number system with base $m$, with identifications $k\overline{m-1}\sim (k+1)\overline{0}$ for $k=0,\dots,m-2$.
\hfill \end{proof}

The case of automata with two self-inverse states is more interesting. Generalizing the tent map example, we find for every $m$ automata generating the symbolic dynamics of zigzag functions on $[0,1]$. For even $m>2$, an edge from $o$ to $c$ has to be added to Figure~\ref{fig2}\textcolor{blue}{b} while the automaton for odd $m$ looks like Figure~\ref{fig1}. Moreover, quite different connected spaces $X$ can be generated.

\begin{example}[An exotic space]\label{ex2}
Figure~\ref{fig5} shows an automaton with two self-inverse states beside $o$ and symbols $0,1,2$. All multiple addresses come as triplets. When we cancel the label $(2,2)$ at $c$, then $X$ will be the dendroid at the top right in Figure~\ref{fig4}. The label $(2,2)$ makes the space more complicated since $X_0\cap X_1\cap X_2$ is a~Cantor set, given by all triple addresses $01s\sim 11s\sim 21s$ with $s\in \{ 0,2\}^\infty $. Actually, $X$ is not homeomorphic to a~subset of the plane, as we shall prove in Proposition~\ref{plan}. Figure~\ref{fig5} shows graph approximations of $X$ where the vertex set is $D^n$ and two vertices $u,v$ are connected by an edge if $(u,v)$ is accepted by the automaton. Such approximations will be discussed in Section~\ref{inver}.
\end{example}

\begin{figure}[h!t]
\begin{center}
\includegraphics[width=0.45\textwidth]{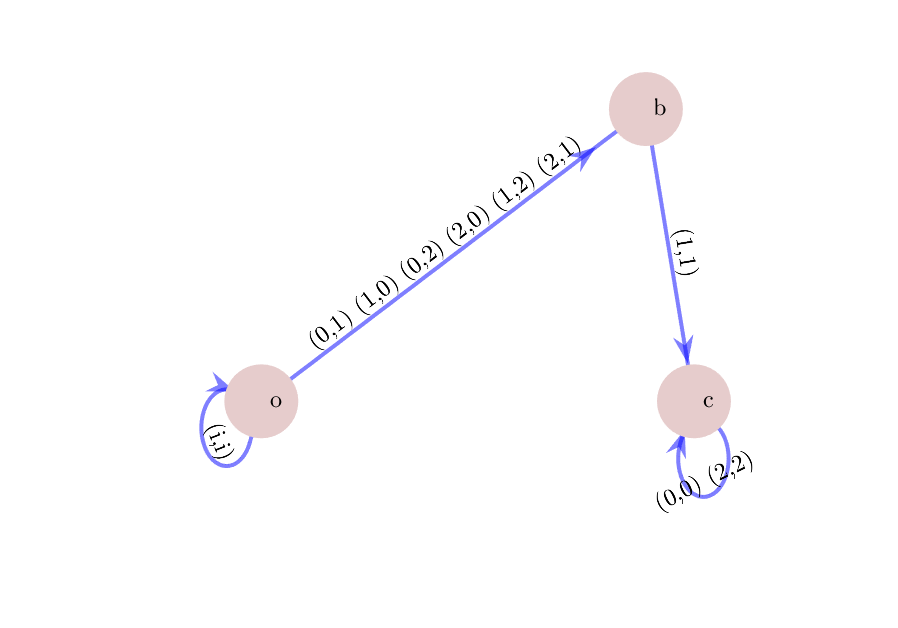} \
\includegraphics[width=0.45\textwidth]{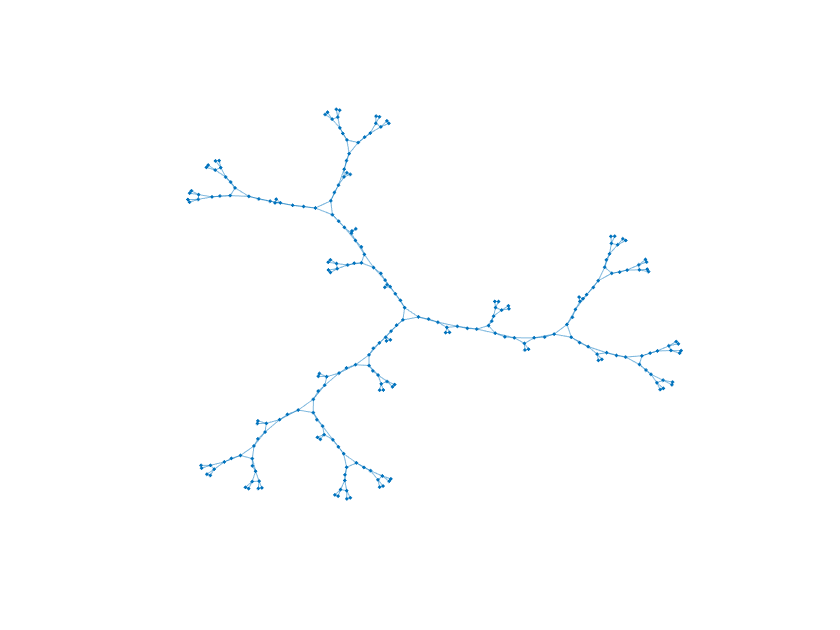} \\
\includegraphics[width=0.45\textwidth]{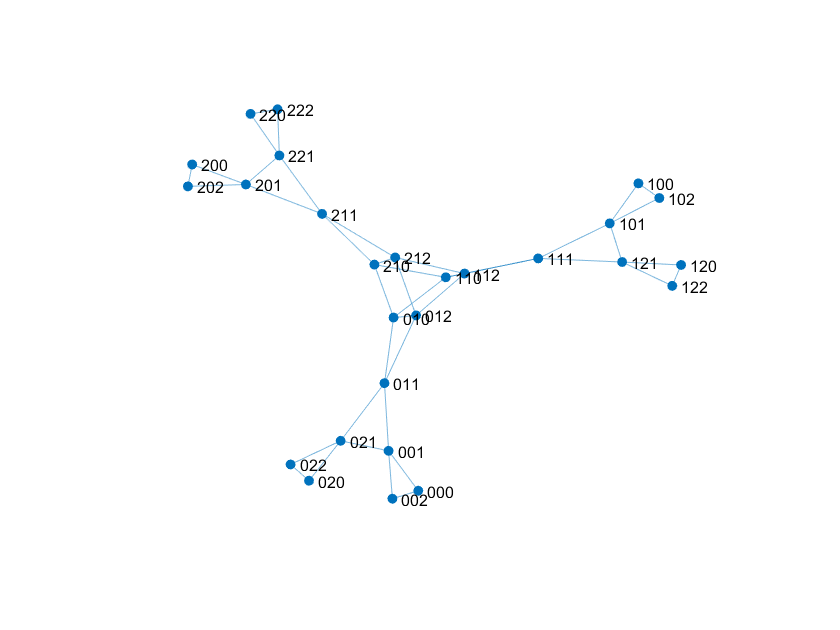} \
\includegraphics[width=0.45\textwidth]{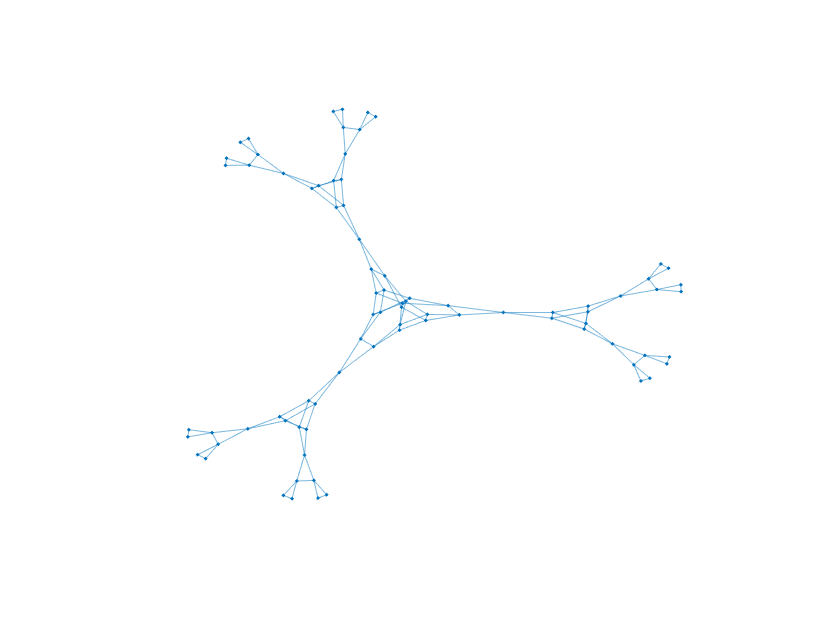} \\
\end{center}
\caption{A complete automaton with 2 self-inverse states and 3 digits. Top right: Without label $(2,2)$ at $c$, a~dendrite $X$ is generated, sketched here on level 5. Bottom: The full space $X$ cannot be embedded into the plane. Graph approximations on levels 3 and 4 are shown.}\label{fig5}
\end{figure}

\section{Topological self-similarity of automata-generated spaces}\label{topsel}
In the definition of a~topology-generating automaton, property 4 required that for each $i\in D$ there is a~loop with label $(i,i)$ at the initial state $o$.

\begin{proposition}\label{prop5}
Property 4 implies that all pairs $(s,s)$ of equal sequences are accepted. Moreover,
\begin{equation}
\mbox{for each }i\in D, \mbox{a pair of addresses }(s,t)
\mbox{ is accepted if and only if }(is,it) \mbox{ is accepted.} \label{tauinv}
\end{equation}
In other words, the equivalence of addresses is invariant to shifting in a~fixed symbol $i$ or word $u$, and to cancelling an initial symbol or word when it is the same in both addresses.

This implies that there are homeomorphisms $h_i:X\to X_i$ such that $h_i\cdot \varphi =\varphi\cdot \tau_i$ where $\tau_i:S\to S$ is defined by $\tau_i(s_1s_2\dots) =is_1s_2\dots\, $ for all $i \in D$. They fulfil the self-similarity equation
\begin{equation}
X =h_0(X)\cup \dots\cup h_{m-1}(X) . \label{sel1}
\end{equation}
\end{proposition}

 \begin{proof} The first assertion concerns the infinite paths through loops at $o$. Now fix $i\in D$. Any path of edges starting in $o$ can be augmented by putting the loop $(i,i)$ in front of the path. On the other hand, if the path starts with the label $(i,i)$, then this must be the loop at $o$ by property 2, and we get another path by cancelling this loop. Thus each equivalence class of addresses is mapped by $\tau_i$ onto another equivalence class. Now $h_i$ is just this mapping $\tau_i$ but acting on equivalence classes instead of single addresses. It is the morphism of the quotient space $X$ induced by $\tau_i$, formally written as $h_i(x)=\varphi\tau_i\varphi^{-1}(x)$ for $x\in X$.

The topological self-similarity is obvious for the symbolic space $S$ and can be written as equation $S=\tau_0(S)\cup \dots\cup \tau_{m-1}(S)$ which describes a~disjoint union. For the quotient space $X$ this implies~\eqref{sel1}
where the intersections of the $X_i=h_i(X)$ are usually not empty. They are subject to the rules of the automaton, however. \hfill \end{proof}

Note that the self-similarity of number systems is just a~matter of convenience. We treat all cylinders $S_u$ of $S$ in the same way, identifying associated pairs of addresses and performing the same operations. It would be much more work to have specific rules for each cylinder. Here we have assumed property 4 which implies~\eqref{sel1}. However, it turns out that some weaker type of self-similarity directly follows from the generation of $X$ by an automaton.

\begin{theorem}[Self-similarity of automata-generated spaces]\label{theo1} Suppose that in Definition~\ref{def1} we replace the property 4 by the requirement that $G$ accepts all pairs $(u,u)$ of equal words. Then the generated space $X$ and its pieces satisfy some graph-directed topological self-similarity equations.
\end{theorem}

 \begin{proof} Let $V_0$ be the set of all states $c$ for which there is a~path of edges starting in $o$ with labels $(i_1,i_1), (i_2,i_2),\dots, (i_n,i_n)$. By property 2, this path will be the only way to accept the pair $(u,u)$ for $u=i_1\dots i_n$. Since we assume also $(ui,ui)$ to be accepted for $i\in D$, each such state $c$ must have outgoing edges with label $(i,i)$ for each digit $i\in D$. Since we do not require loops, this means that each $c\in V_0$ can be taken as the initial state of the automaton and will then generate a~topological space $X^c$ and an address map $\varphi_c: S\to X^c $. The edges of $V_0$ will determine how these spaces and maps depend on each other.

\begin{figure}[h!t]
\begin{center}
\includegraphics[width=0.47\textwidth]{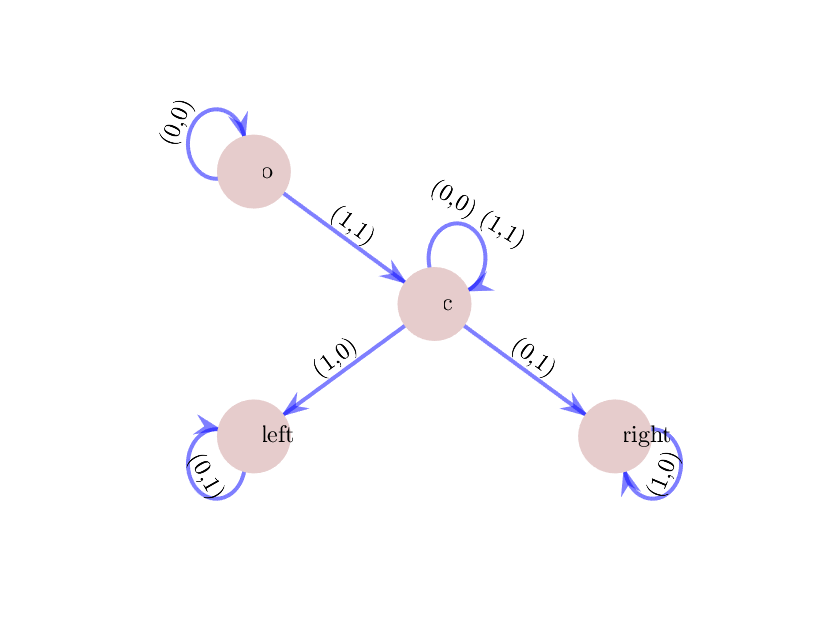} \
\includegraphics[width=0.47\textwidth]{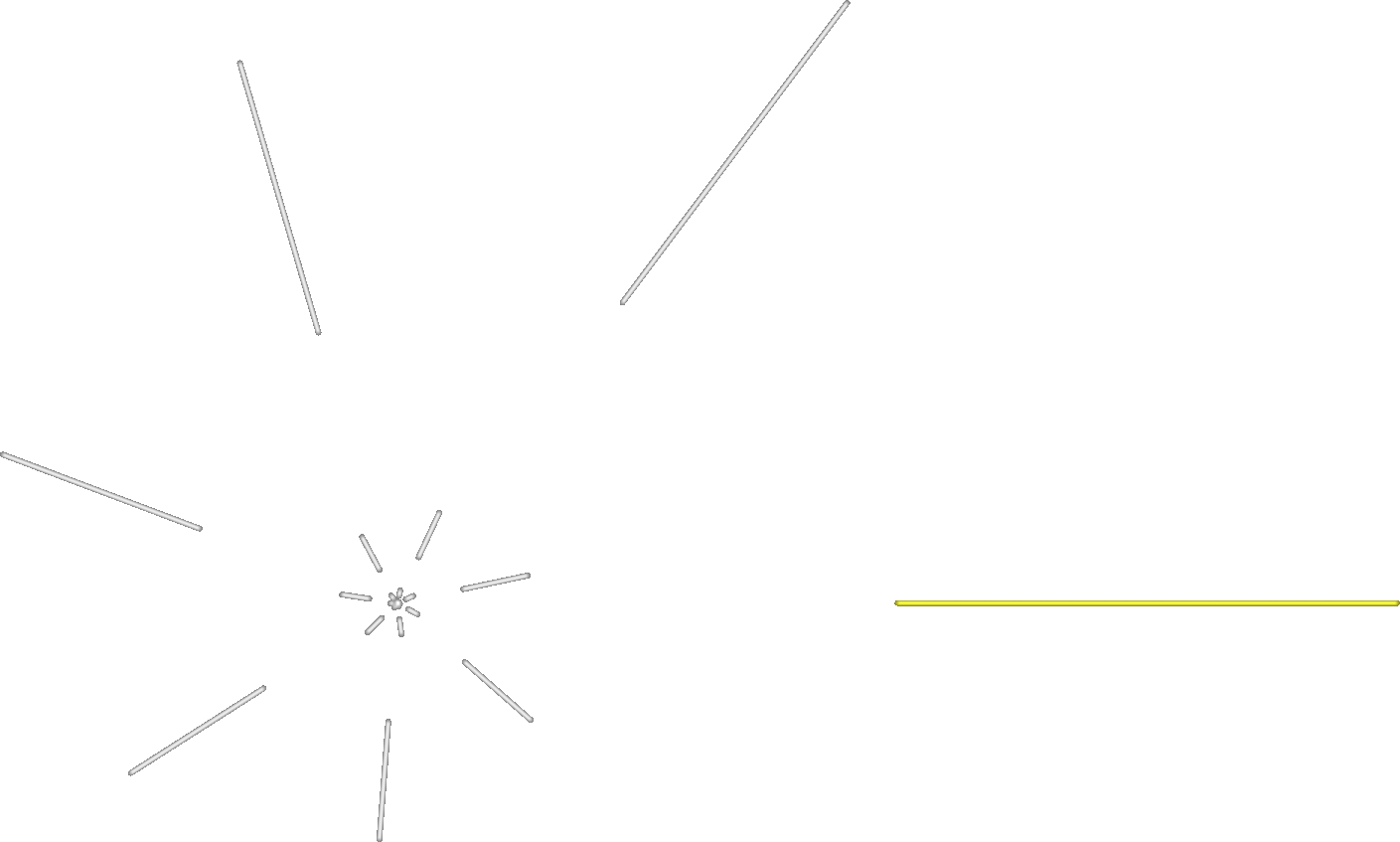}
\end{center}
\caption{Graph-directed topological self-similarity for an automaton without property 4. When we start in state $c$, we obtain an interval. When we start in $o$, the space $X$ consists of a~sequence of intervals $X_1, X_{01}, X_{001},\dots$ and a~limit point with address $\overline{0}$. The spiral arrangement of intervals was chosen for good visibility.}\label{fig6}
\end{figure}

\begin{example} Figure~\ref{fig6} shows a~very simple example with $V_0=\{ o,c \} $. The outgoing edges from $c$ are as in Figure~\ref{fig1} so that with initial state $c$ we get the interval $X^c =[0,1]$ and $\varphi_c:S\to X^c$ identifies the addresses of binary numbers. Proposition~\ref{prop5} holds because of the two loops at $c$. So we have two homeomorphisms $h^c_i:X^c \to X^c_i$. Equation~\eqref{sel1} becomes $X^c =h^c_0(X^c)\cup h^c_1(X^c)$.

At the initial state $o$ we have no outgoing edges labelled $(i,j)$ with $i\not= j$ so that our Definition~\ref{def1} with loop $(1,1)$ at $o$ would yield no identifications and $\varphi :S\to X=X^o$ would be the identity map. Indeed $X_0$ and $X_1$ are disjoint. However, the edge with label $(1,1)$ from $o$ goes to $c$. It says that there are identifications in $X_1$ exactly as in $X^c$. Thus $X_1$ is an interval, and so is $X_{0^k 1}$ for $k=1,2,\dots$, as seen in Figure~\ref{fig6}. Actually, $X$ has almost the same identifications as $X^c $. Only the links $0^k 0\overline{1}\sim 0^k 1\overline{0}$ between successive intervals are missing for $k=0,1,2,\dots\, $.

Since identifications in $X^c$ are the same as in the piece $X_1$, there is a~natural homeo\-morphism $h_1:X^c \to X_1\subseteq X$ which can be written as $h_1(y)=\varphi\tau_i\varphi_c^{-1}(y)$ for $y\in X^c $. Note that $h_1$ corresponds to the edge from $o$ to $c$ and maps into the opposite direction. Moreover, identifications in the piece $X_0$ are the same as in $X$, so as in Proposition~\ref{prop5} we have the homeomorphism $h_0$ from $X$ onto $X_0$. Summarizing we get a~system of set equations
\[
X =h_0(X)\cup h_1(X^c)\quad, \quad X^c =h^c_0(X^c)\cup h^c_1(X^c)
\]
which by definition expresses the topological graph-directed self-similarity~\cite{MW} of the two spaces $X, X^c $. One realization is shown in Figure~\ref{fig6} where $h_0$ includes a~rotation for better visibility.
\end{example}

In the general case there is one quotient space $X^c$ for each possible initial state $c\in V_0$. And we said that for each $c$ and each digit $i\in D$ we have a~unique edge $e(c,i)$ starting in $c$ with label $(i,i)$, with a~uniquely determined endpoint $d(c,i)\in V_0$. Thus starting in $c$ with digit $i$, we will make the same identifications as in $X^{d(c,i)}$. Thus there is a~homeomorphism $h_i^c :X^{d(c,i)}\to X_i^c $ onto the piece $X_i^c$ of $X^c $. This yields the equations
\begin{equation}\label{sel2}
X^c = \bigcup_{i\in D} X_i^c = \bigcup_{i\in D} h_i^c (X^{d(c,i)})\quad\mbox{ for } c\in V_0.
\end{equation}
This system of equations expresses the graph-directed self-similarity as introduced by Mauldin and Williams~\cite{MW} and other authors~\cite{Ba3,Bar,Edgar,Gi87}
for contractive similitudes $h_i^c$ in metric spaces. Contractive maps are needed to prove that there is a~solution consisting of compact sets $X^c, c\in V_0$. In our case, the solution is constructed by the automaton. So we can consider the more general case of homeomorphisms.
 \hfill \end{proof}

In this note, we assumed the stronger property 4 because we think that the graph self-similarity should be better discussed in a~more comprehensive setting where the symbolic space is a~sofic subshift instead of a~full shift. This chapter has shown that the topology of our spaces always comes together with coverings by pieces on different levels. This is part of their structure and a~consequence of their generation by automata.

This raises mathematical questions which we only mention. What is the appropriate isomorphism concept for automata-generated spaces? Homeomorphy is to wide, Lipschitz equivalence~\cite{luoliu,ruan17,ZR16,ZY18} seems too narrow. What about quasiconformal maps? Or shall we better speak about spaces with a~graded block structure?

 \section{More examples}\label{exam}
The previous section shows that self-similar fractals are typical examples of automata-generated spaces. Now we shall see that most examples from fractal geometry have a~simple automatic structure. Self-affine tiles will not be mentioned since their automata are well described in~\cite{ALT,DN05,LL13,LL07,Lo,LZ17,ScT,TZ20}.

\begin{figure}[h!t]
\begin{center}
\includegraphics[width=0.3\textwidth]{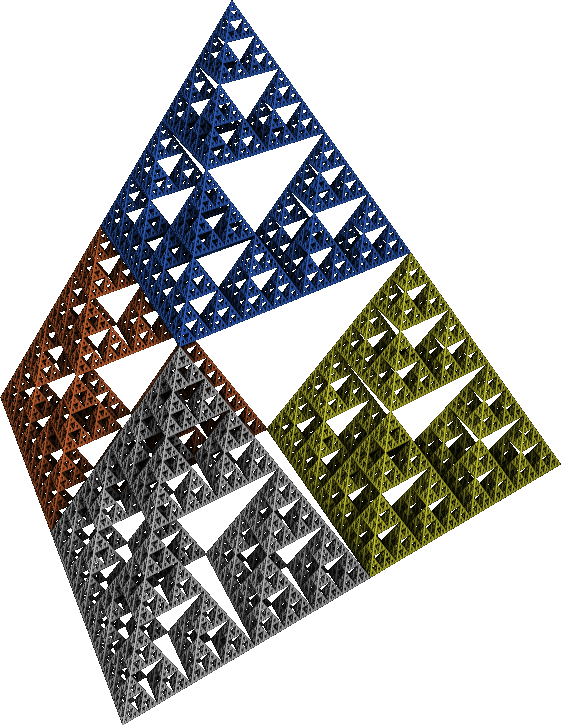} \
\includegraphics[width=0.65\textwidth]{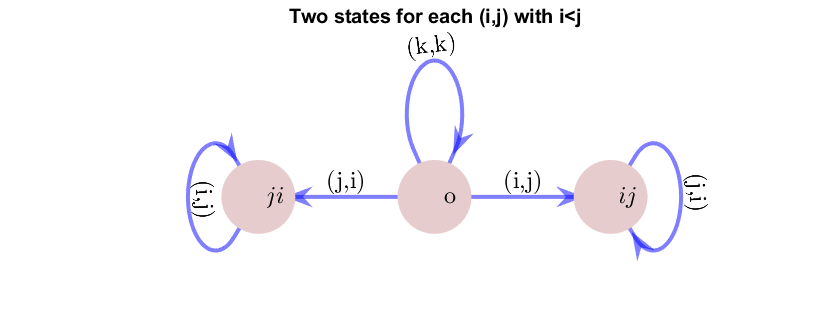}
\end{center}
\caption{Sierpiński tetrahedron and its automaton, $0\le i<j\le 3$.}\label{b1}
\end{figure}

\paragraph*{Sierpiński gasket and tetrahedron. } The digits are $0,1,\dots,n$ for the $n$-dimensional version. For $n=1$ we have Example~\ref{ex1}. There are states $ij$ for all pairs of digits $i\not= j$. Thus in the plane we have 6 states and in $n$ dimensions we have $n(n+1)$. There are only edges from $o$ to all other states and loops, exactly as in Example~\ref{ex1}. See Figure~\ref{b1}.

\begin{figure}[h!t]
\begin{center}
\includegraphics[width=0.45\textwidth]{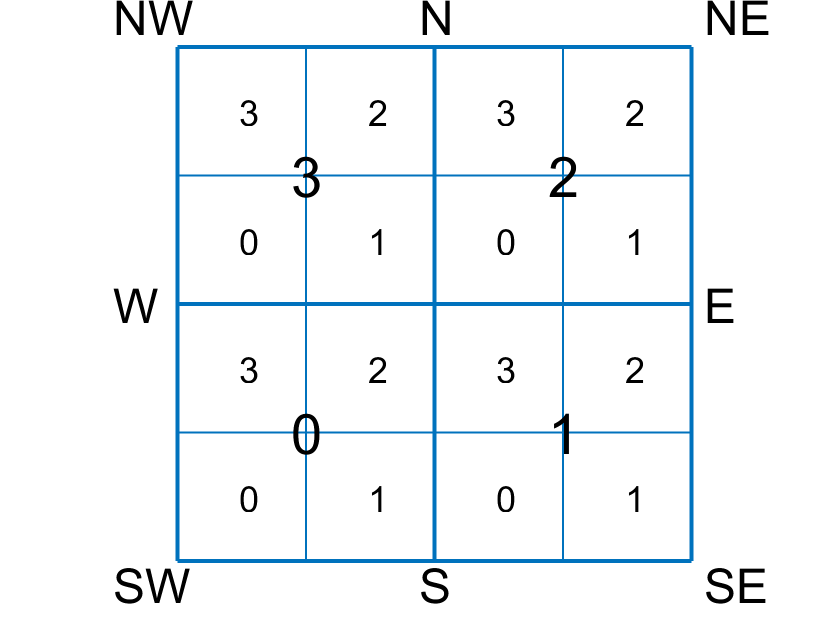} \
\includegraphics[width=0.45\textwidth]{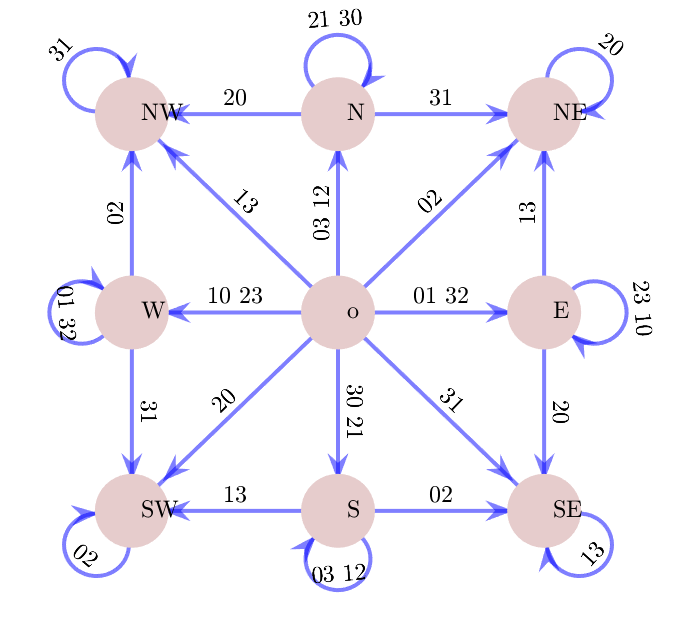}
\end{center}
\caption{The complete automaton for the $2\times 2$ square. With more labels, the graph of this automaton applies to the $k\times k$ square, and also to fractal squares.}\label{b2}
\end{figure}

\paragraph*{The square. } The $2\times 2$ square is the basis of commonly used `quadtree' methods for image coding and processing. We use digits $0,1,2,3$ as indicated in Figure~\ref{b2}. There are four states $N,E,S,W$ for the four sides of the square and four states $NE,SE,SW,NW$ for the vertices. This can be considered as a product of Example~\ref{ex1} with itself. A $k\times k$ square has $k^2$ digits, but the same states and edges. Cubes are also easy to construct.

\begin{figure}[h!t]
\begin{center}
\includegraphics[width=0.37\textwidth]{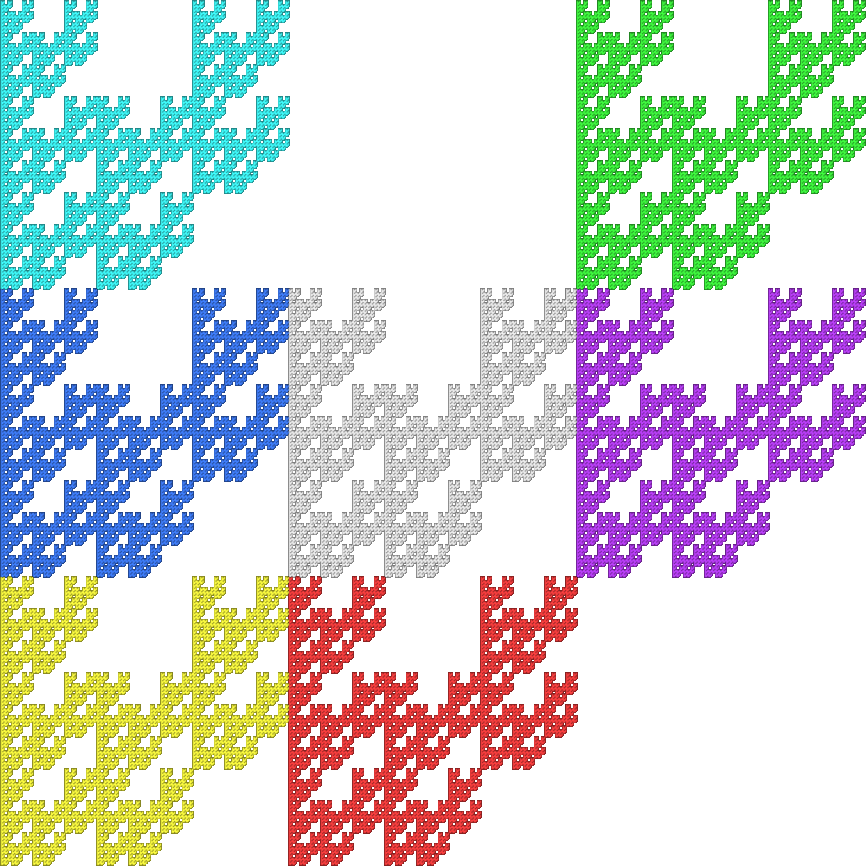} \qquad
\includegraphics[width=0.4\textwidth]{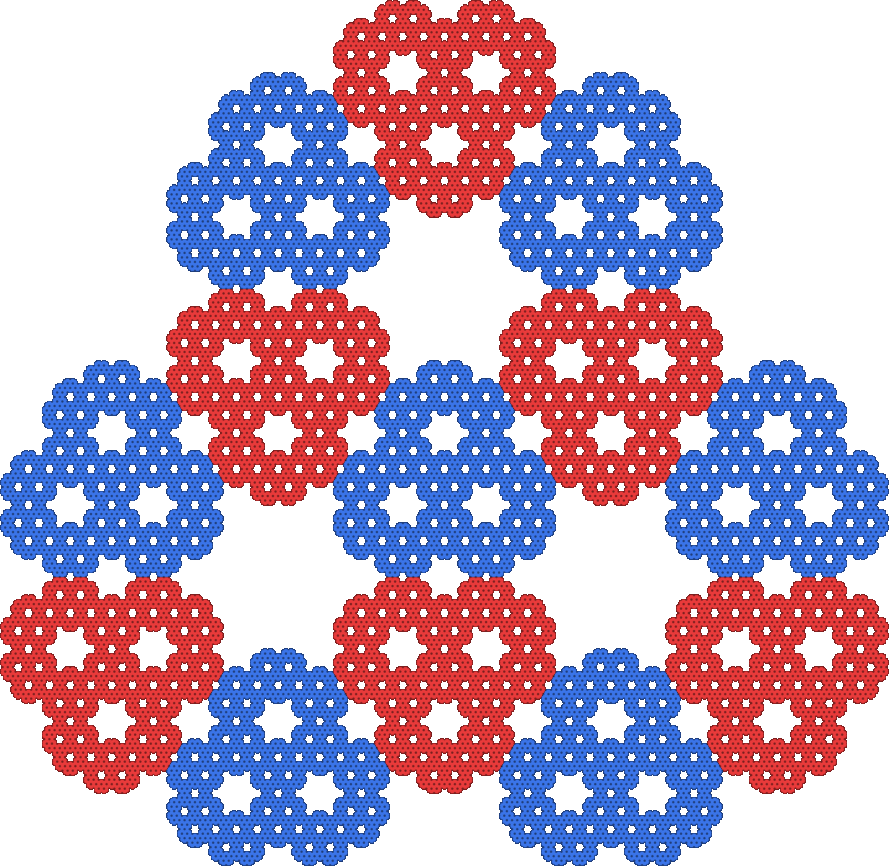}
\end{center}
\caption{A fractal square which does not admit states $NW$ and $SE$, and a~fractal triangle representable with 13 digits and only 3 states.}\label{b3}
\end{figure}

\paragraph*{Fractal squares and triangles. } This class of examples is obtained from the $k\times k$ square by taking less than $k^2$ digits. The classical example is the Sierpiński carpet, with a~hole in the middle of a~$3\times 3$ square. There are a~lot of recent papers on fractal squares, mainly by Chinese authors. See~\cite{DT23, huangrao, luoliu, ruan17, RWX22, xi20, xiao21, ZR16} and their references.
The automaton for fractal squares is the same as for the corresponding square, with fewer edge labels. Some states can disappear, as in Figure~\ref{b3}. The case of only two states was studied by Zhu and Yang~\cite{ZY18} with automata similar to ours.

Figure~\ref{b3} also shows an example of a~`fractal triangle'. Every second triangle is put upside down together with its sides. The automaton for this example needs three states associated with the three sides of a~triangle, and has two loops at each state. Related examples will need six states when triangles meet at their vertices.

\paragraph*{Post-critically finite sets}\cite{BK,Ka93,Kig,T85}. Consider an address map $\varphi: S\to X$ with
\begin{equation}
\varphi (s)=\varphi(t) \Longleftrightarrow \varphi (is)=\varphi(it)\quad\mbox{ for each digit } i .
\label{addsel}
\end{equation}
This is the automata-free formulation of self-similarity. In the setting of Proposition~\ref{prop5}, it implies the equation~\eqref{sel1} for the homeomorphisms $h_j(x)=\varphi^{-1}\tau_j\varphi (x), j=0,\dots,m-1$ from $X$ into $X$. The address map and space $X$ are called post-critically finite if there are finitely many double addresses $(s,t)$ with $s_1\not= t_1$, and all these sequences are preperiodic. This holds for Figures~\ref{fig1}--\ref{fig4} and the Sierpiński tetrahedron, but not for the square, Example~\ref{ex2} and Figure~\ref{b3}. The following result is related to~\cite[Proposition 6]{BM09} and~\cite{RW23}.

\begin{proposition}
 
For a~map $\varphi: S\to X$ with~\eqref{addsel}, the following conditions are equivalent.
\begin{enumerate}
\item[ (i)] $\varphi$ is post-critically finite,
\item[(ii)] There is an automaton $G=(V,D^2,E,o)$ accepting the double addresses of $\varphi $, where there is no directed path between two directed cycles of the graph $(V,E)$.
\end{enumerate}
\label{pcf}
\end{proposition}

 \begin{proof} If $\varphi$ is postcritically finite, each double address $(s,t)$ with $s_1\not= t_1$ can be written as $s=s_1\dots s_n\overline{u}$ and $t=t_1\dots t_n\overline{v}$ with a~common $n$ and words $u=u_1\dots u_k, v=v_1\dots v_k$ of equal length. Consider a~path of directed edges $e_j, j=1\dots n+k$ with labels $(s_j,t_j)$ for $j=1,\dots,n$ and $(u_i,v_i)$ for $j=n+i, i=1,\dots,k$. Let the terminal point of the last edge be the initial point of $e_{n+1}$.
Doing this for all double addresses with $s_1\not= t_1$ and joining the paths at their initial point which we call $o$, we get the automaton.

Now assume the automaton is given and there is no path between two directed cycles. In particular, cycles are disjoint. Then from $o$ to any cycle there are only finitely many paths of edges. Thus each cycle accepts only finitely many pairs of preperiodic addresses. Since each infinite path from $o$ must go through a~cycle, this implies that $\varphi$ is postcritically finite. \hfill \end{proof}

\section{Automata for multiple addresses}\label{comp}
Definition~\ref{def1} does not require that all pairs of addresses $(s,t)$ of one point $x$ in $X$ are accepted by the automaton. Such pairs can also come from the transitivity of the equivalence relation, say when $(s,s')$ and $(t,s')$ are accepted by the automaton for some sequence $s'$. In Figure~\ref{fig4} we added one more state and edge for completeness. However, this was not necessary. The incomplete automaton provides the correct quotient space $X$.

The point here is that the binary relation defined by a~topology-generating automaton on words of length $n$ need not be transitive: $\varphi(S_u)\cap \varphi(S_{u'}) \not=\emptyset$ plus $\varphi(S_v)\cap \varphi(S_{u'}) \not=\emptyset$ does not imply $\varphi(S_u)\cap \varphi(S_v) \not=\emptyset $. However, for $n\to \infty$ the pieces $S_u$ shrink down to a~point, and so do their images. So for sequences, the transitivity of the automata relation is included in the definition: $\varphi(s)=\varphi(s')$ plus $\varphi(t)=\varphi(s')$ implies $\varphi(s)=\varphi(t)$.

\begin{figure}[h!t]
\begin{center}
\includegraphics[width=0.4\textwidth]{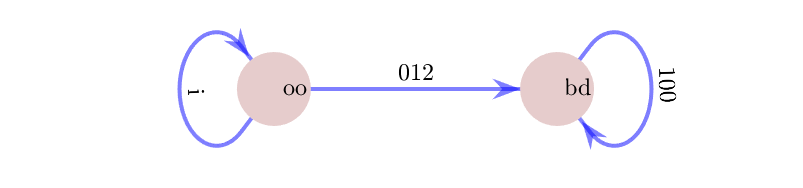} \quad
\includegraphics[width=0.5\textwidth]{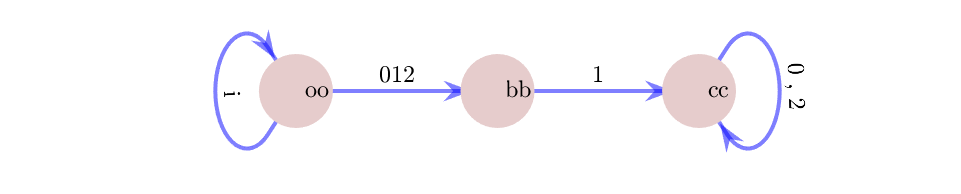}
\end{center}
\caption{In Figures~\ref{fig4} and~\ref{fig5} we have only triple addresses and no proper double addresses. Here are the automata $G_3$ for the triple addresses which can be constructed by the method in Section~\ref{comp}. Label $iii$ is abbreviated as $i$.}\label{figthree}
\end{figure}

Usually, incomplete automata are much simpler than the complete version, as shown by Figure~\ref{figsquare} below compared to Figure~\ref{b2} above. And given an automaton, we do not know whether it is complete or incomplete. We now describe an abstract algorithm that determines automata for triple and multiple addresses from a~given automaton $G$. It decides the completeness of $G$ and also finds cases where we have no proper double addresses, only multiple addresses, as in Figure~\ref{figthree}. In some way, we unfold the automaton $G$, generalizing Gilbert's work on triple addresses of complex number systems with base $-n+i$~\cite{Gi82}.

\begin{theorem}[Automata for triple and multiple addresses]\label{theo2} Let $G=G_2$ be a~topology-generating automaton with finite equivalence classes, with size smaller than a~constant $C$. Let $\varphi: S\to X$ be the corresponding address map. All multiple addresses of points in $X$ can be determined by automata $G_3, G_4,\dots$ derived from $G$, without any geometric knowledge of $X$. The automaton $G_k$ will accept all $k$-tuples of equivalent addresses that are not contained in a~$(k+1)$-tuple of equivalent addresses.
\end{theorem}

 \begin{proof} For triple addresses, we have a~product construction: $G_3$ will be a~subset of $G_2~\times~G_2=(V^2,D^3,E_3,oo)$ where the states are denoted $bc$ with $b,c\in V$. There is an edge $(bc,b'c')\in E_3$ with label $ijk$ if $(b,b')$ and $(c,c')$ are edges in $G_2$ with labels $(i,j)$ and $(j,k)$, respectively. Two linked pairs of digits are coupled at their common item, which for paths will then realize the transitivity of the equivalence relation. To obtain $G_3$, we omit all vertices from $G_2\times G_2$ without outgoing edge and those components of the graph on which there are no edge paths with three different addresses. This leads from Figures~\ref{fig4} and~\ref{fig5} to Figure~\ref{figthree}, for instance.

Now we describe the construction of $G_4,G_5$, etc.
The algorithm for extending $G_n$ to $G_{n+1}$ is based on taking products $G_n\times G_2$. That is, we join states $c$ of $G_2$ to $(n-1)$-tuples $b=b_1\dotsb_{n-1}$ which describe states of $G_n$. We draw edges when the last coordinate $j$ of the label between $b$ and $b'$ agrees with the first coordinate of the label
$(j,k)$ between $c$ and $c'$. In this way, we extend the $n$ digits at certain edges of $G_n$ by another digit. If this works for an infinite path of edges, we have extended the number of addresses from $n$ to $n+1$. If no further address can be joined to an equivalence class, it is complete and belongs to $G_n$. If we can add an address, we get a~class which belongs to $G_{n+1}$. Since the size of equivalence classes is assumed to be bounded, the algorithm stops after finite time.

Two problems must be mentioned. On the one hand, the new address which we find may often coincide with a~previous address. The cleaning of the graph
needs lexicographic ordering and recursive procedures which we do not discuss. On the other hand, it is not enough to take only the last entry of the labels of $G_n$ as $j$. For the examples below, a~new address has sometimes to be appended at the first entry.
In general, we cannot assume that transitivity is realized in a~chain. We must check the product with $G_2$ for each fixed coupling coordinate $j$ between 1 and $n$, and then take the union of the $n$ graphs with their initial states identified. This is possible since in any instance we add at most one new address. In the general case, there is no meaning in writing the states of $G_n$ as tuples of states in $G_2$. We can take any other names. There is another simplification in extending $G_n$ at each coordinate $j$. We need only one permutation of each set of addresses in $G_n$, because $G_2$ satisfies property 3 in Definition~\ref{def1}. Thus the graphs $G_k$ with $k\ge 3$ are represented in asymmetric form, simpler than $G_2$, as can be seen in Figure~\ref{figthree}. Only when a~sequence of edges leads back to a~state with a permutation of the addresses, we have to replicate the state.

Any address that is equivalent to one address in an $n$-tuple of addresses in $G_n$ will be found in this way. And if such a~new address does not exist, the $n$-tuple is confirmed as being complete and belonging to $G_n$. Every incomplete address will be extended and eventually be completed since $n$ is bounded by assumption. This concludes the proof.
\hfill \end{proof}

The method must be augmented by ordering and graph cleaning procedures for efficient programming. In the present form, it works well for small examples, as in Figure~\ref{figsquare} below and Example~\ref{ex4} in Section~\ref{IFSex}. We now discuss a~more complicated case.

\begin{example}[An incomplete automaton generating a~triangle]\label{ex3}
Figure~\ref{fig7} shows an automaton with three states and three digits.
The corresponding space $X$ is a~well-known triangle. It is a~self-similar set and the fundamental domain of the Coxeter group $\tilde{G_2}$ generated by the reflections at the three sides of the triangle. The neighbor graph mentioned in Section~\ref{IFSex} is a~complete automaton describing all pairs of equivalent addresses. For this example it has 16 states and 42 arrows which explains why we prefer the incomplete automaton.d
Some multiple addresses can be determined by eyesight from the geometry of triangle pieces. Note that two vertices of the triangle have addresses $\overline{0}$ and $\overline{2}$ and all multiple addresses will end with one of these suffixes. There are points with 4, 6, and even 12 addresses which correspond to the vertices of the corresponding crystallographic tiling.
\end{example}

\begin{figure}[h!t]
\begin{center}
\includegraphics[width=0.47\textwidth]{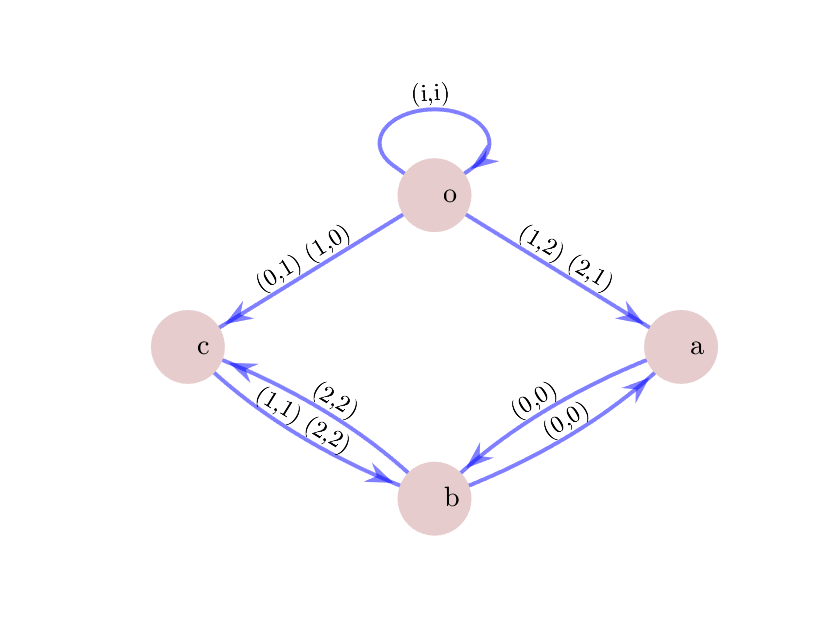} \
\includegraphics[width=0.47\textwidth]{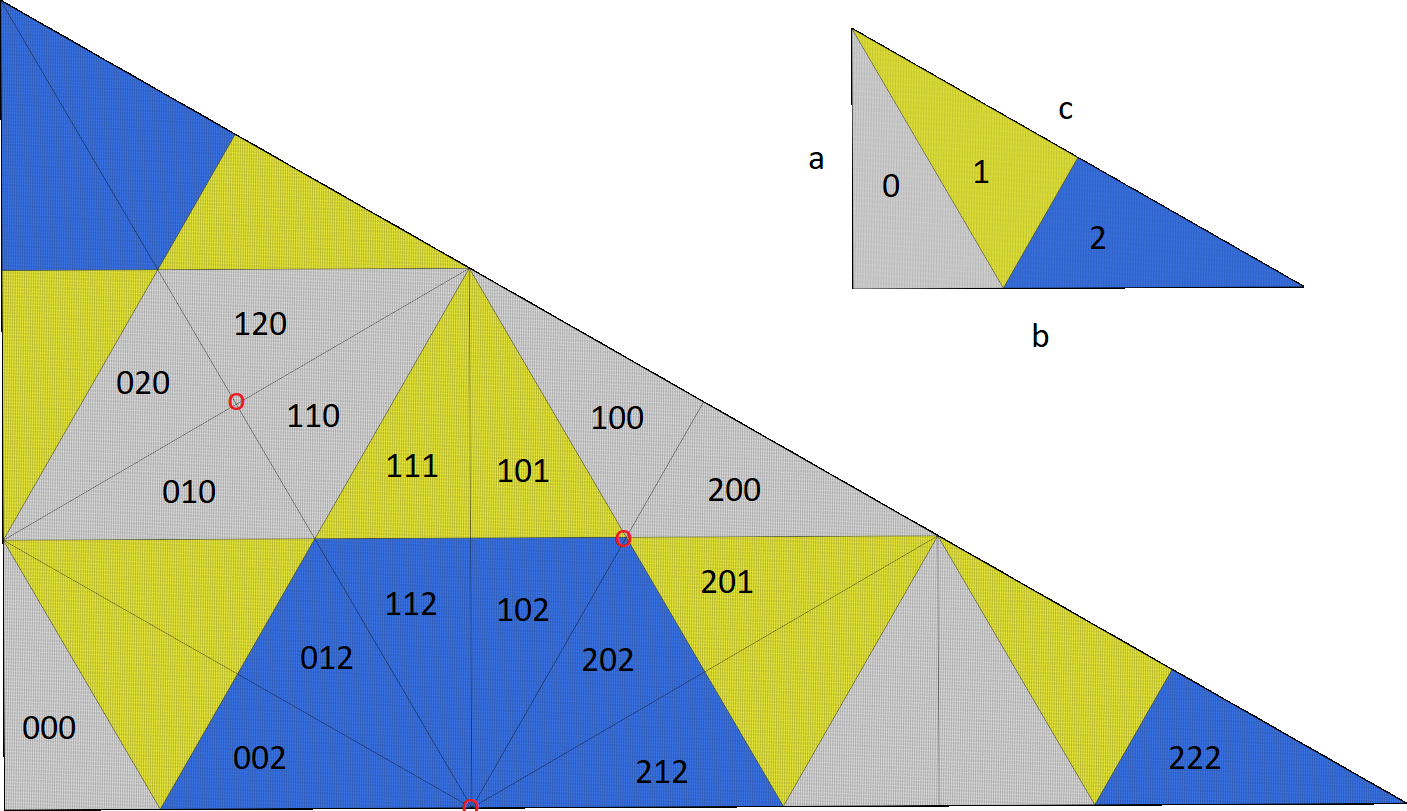}
\end{center}
\caption{An incomplete automaton generating a~self-similar triangle. The states of the automaton correspond to the three sides $a,b,c$. In the small triangle $0c$ is identified with $1c$ and $1a$ with $2a$. The big triangle shows third-level pieces around points with 4, 6, and 12 addresses.}\label{fig7}
\end{figure}

We construct $G_3$ for Example~\ref{ex3} by taking $G_2\times G_2$. The states $ac$ and $ca$ have no outgoing edges. Moreover, the states $oo, aa, bb, cc$ are connected by edges with the first label equal to the third. Since $aa, bb, cc$ do not lead to other states, they can be omitted. Thus from 16 states, we are left with $oo$ and 10 other states of $G_3$. Because of symmetry, we do not show the states $oa,ob,oc$ in Figure~\ref{fig8}. Labels $iii$ are abbreviated $i$.

We have paths of two edges from $oo$ to $bc$ and $ba$ which directly yield triple addresses. Paths through $ao$ yield $10\overline{2}\sim 20\overline{2}\sim 21\overline{2}$ and $20\overline{2}\sim10\overline{2}\sim 11\overline{2}$ for the two labels of the edge $(oo,ao)$. Similarly, the paths from $oo$ over $co$ to $bc$ yield $01\overline{2}\sim 11\overline{2}\sim 10\overline{2}$ and $11\overline{2}\sim 01\overline{2}\sim 00\overline{2}$. For the paths to $ba$ two pairs of labels can be combined: $01\overline{0}\sim 11\overline{0}\sim 12\overline{0}$ and $02\overline{0}\sim 12\overline{0}\sim 11\overline{0}$ and
$11\overline{0}\sim 01\overline{0}\sim 02\overline{0}$ and finally
$12\overline{0}\sim 02\overline{0}\sim 01\overline{0} $. In Figure~\ref{fig7}, these triples correspond to three consecutive gray third-level pieces at the point marked between pieces 0 and 1. The other four triple addresses correspond to consecutive blue triangles around the marked point on the basic line.
We see that the triples are not complete: the paths to $ba$ correspond to a~quadruple and the paths to $bc$ to a~$6$-tuple of addresses. Thus we do not need $G_3$, but $G_4$ and $G_6$.

\begin{figure}[h!t]
\begin{center}
$G_3$\includegraphics[width=0.46\textwidth]{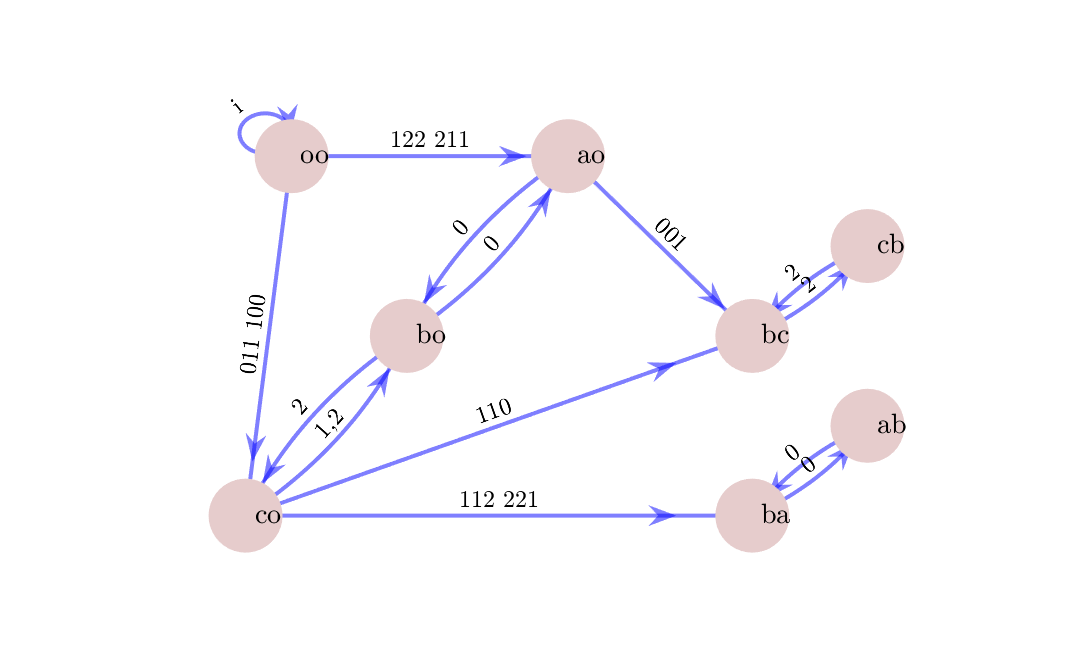}
$G_4$\includegraphics[width=0.46\textwidth]{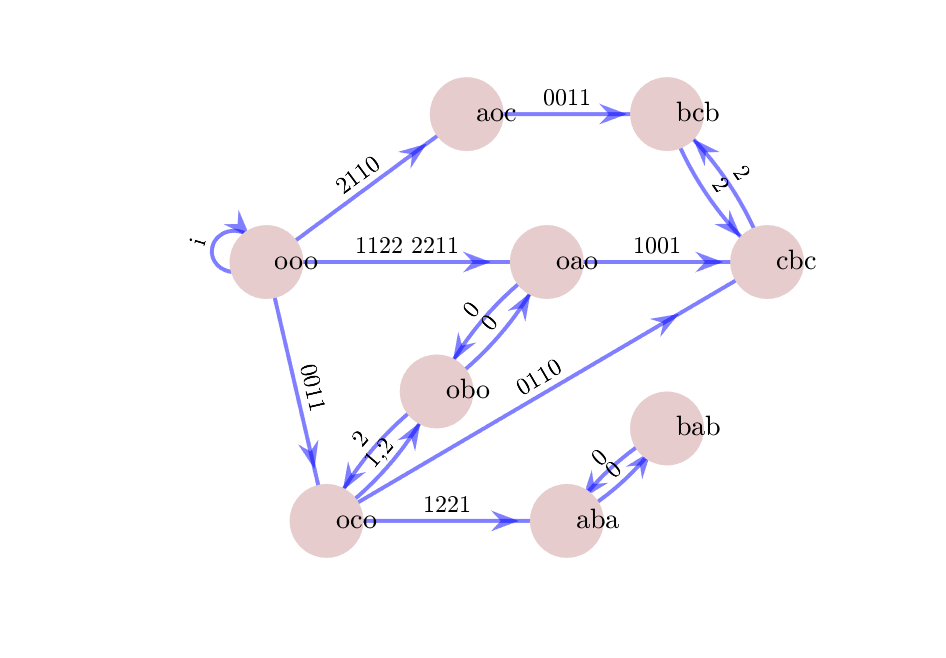}
\end{center}

\caption{Automata $G_3$ and $G_4$ for example~\ref{ex3}. Ignoring loops at $oo$ and paths through $bo$, there are eight triple addresses that correspond to the common endpoints of three consecutive third-level triangles in Figure~\ref{fig7}. These triples can all be extended to quadruples. The quadruple addresses leading to $aba$ in $G_4$ are complete, those leading to $cbc$ can be further extended.}\label{fig8}
\end{figure}

We construct $G_4$ starting with the edge from $oo$ to $a0$ in $G_3$. For 122 a~new item can be appended in front, but not at the end. However, 211 can be extended to both sides, yielding edges from $ooo$ to $aoc$ and $oao$. The first edge determines a~single path while the second one determines the labeling of the rest of the network, as shown on the right of Figure~\ref{fig8}. All other couplings of addresses, for instance appending to 122 or taking the other label of edges to and from $c0$ in $G_3$, will lead to repetitions, up to permutation of the addresses.

All triple addresses have been extended, so $G_3$ is now irrelevant. In $G_4$ in Figure~\ref{fig8} the paths leading to $aba$ describe quadruples that are complete. There are two possible digits for the starting edge at $ooo$ and two possible digits for the final edge at $aba$, so there can be only four addresses. When we try to append another address, it will be a~repetition. Thus the part of $G_4$ which leads to $aba$, comprising six vertices, is the graph of complete quadruple addresses.

The three paths leading to $bcb$ and $cbc$ have still incomplete labels. On the upper path, we can append on both sides. On the path in the middle, we can append two digits on the right of 2211. In both cases we get the six-tuple path
\[
G_6 :  ooooo \xrightarrow{221100} oaoco \xrightarrow{100110} cbcbc \stackrel{\mathrm{2}}{\longleftrightarrow} bcbcb \, .
\]
This path describes the six addresses of the marked point on the basic line, and, together with the loops at the initial state, all points in $X$ with six and more addresses. Thus this is already the graph $G_6$. The points that lie on the sides of the triangle $X$ have exactly six addresses, the other have twelve. The sides of type $b$ on the boundary are $b,0b,1b,2b$. We have the self-similarity equation $b=00b\cup 21b\cup 22b $. So the boundary has an automatic structure, with address language $\{ 00,21,22\}^* \{\lambda, 0,1,2\}$ where $\lambda$ is the empty word. Thus we can replace $ooooo$ in $G_6$ by a~small automaton with five states to describe exactly the complete six-tuple addresses of points on the boundary. This is not contained in our general method which also produces a~larger version of $G_6$.

To determine the twelve-tuple addresses, we really need the states with several $o$ as in Figure~\ref{fig8}. We have to amalgamate two six-tuple addresses along the line $1a=2a$ (marked point in Figure~\ref{fig7}) or along $0c=1c$ (two points on both sides of the marked point). Here is one version of the automaton, with notation $v^n =vv\dots v$. All twelve tuples are complete.

\begin{small}
\begin{eqnarray*}
G_{12}:  \textstyle o^{11} \ \ \xrightarrow{\mathrlap{1^6 2^6}\phantom{1^6 2^6}}& o^5 ao^5 & \textstyle \xleftrightarrow{\mathrlap{0}\phantom{0}}\ o^5 bo^5 \ \xrightarrow{\mathrlap{(221100)^2}\phantom{(221100)^2}} oaocoaocoao \ \xrightarrow{\mathrlap{(1001)^3}\phantom{(1001)^3}} cbcbcbcbcbc \ \xleftrightarrow{\mathrlap{2}\phantom{2}} bcbcbcbcbcb\,. \\ 
0^6 1^6 \searrow & o^5 co^5 & \mbox{\small 2}\swarrow \nearrow\mbox{\small 1,2}
\end{eqnarray*}
\end{small}

This completes the discussion of Example~\ref{ex3}. The simpler case of the square is illustrated in Figure~\ref{figsquare}. Example~\ref{ex4} with Figure~\ref{dog} leads to $G_3, G_4$, and $G_6$.

\begin{figure}[h!t]
\begin{center}
\includegraphics[width=0.3\textwidth]{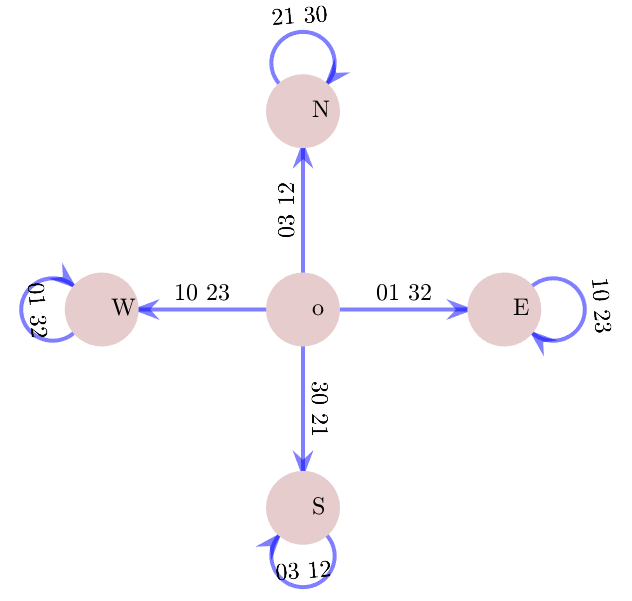} \quad
\includegraphics[width=0.6\textwidth]{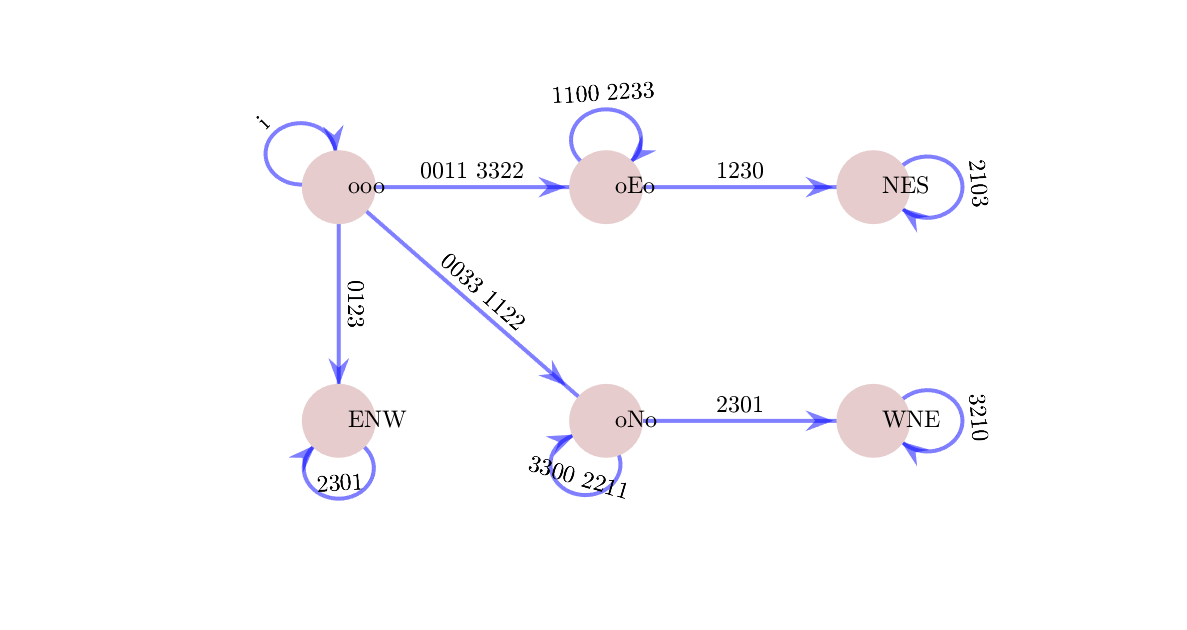} \\
\includegraphics[width=0.7\textwidth]{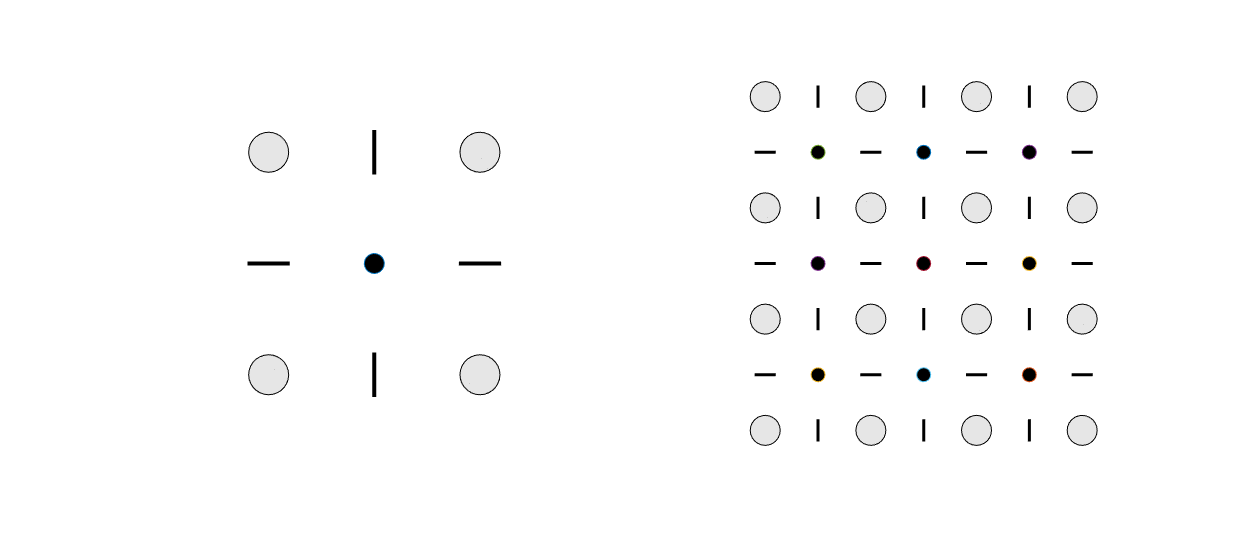}
\end{center}
\caption{Top: a~simple incomplete automaton for the $2\times 2$ square and the corresponding automaton of 4-tuple addresses. Bottom: the finite topological spaces $X_1$ and $X_2$ approximating the square $X$. Gray circles are open points corresponding to $D$ and $D^2 $. Black circles are closed points representing the 4-tuple addresses, and line segments correspond to double addresses. The automata organize this pattern~\cite{kov} in a~hierarchical way. }\label{figsquare}
\end{figure}

\section{Construction of the topological space}\label{inver}
Given an automaton $G$, the definition of the space $X$ as quotient space does not provide much insight. We now present a~meaningful construction of $X$ by approximation from the edge paths in $G$. The approximating spaces look very much like graphs, as in Figure~\ref{fig5}. They are finite topological spaces. A topology on a~finite set is defined by assigning to each point $x$ a~minimal open neighborhood $U_x$~\cite[p. 2]{Bam}.

Neglected for decades, finite topological spaces have gathered a~lot of attention in recent years, due to applications in image processing~\cite{kov} and topological data analysis. Some parts of algebraic topology can be simplified by their use~\cite{Bam, Fern}. In our context, they form the proper tool to understand automata-generated spaces.

\begin{definition}[Approximations of automata-generated spaces]\label{def2}
Let $G$ denote a topology-generating automaton with finite equivalence classes, with size smaller than a~constant $C$, and $X$ the generated topological space. Let $G_k, k\in K$, be the automata of complete $k$-tuple addresses, as constructed in Section~\ref{comp}. For $k\in K$ and each level $n=1,2,\dots$ let $E^n_k$ denote the set of directed paths of $n$ edges in $G_k$ starting in $o$. Each $y\in E_k^n$ corresponds to a~set $W_y\subset D^n$ of $k$ words of length $n$ which form the labels of the path and are accepted at the endpoint of $y $. Let
\[
E^n =\bigcup_{k\in K} E_k^n \quad \mbox{ and }\quad X^n =D^n \cup E^n \
\]
with the following topology. Each point $x\in D^n$ is open. For a~path $y\in E_k^n$ the minimal neighborhood of $y$ must contain $y$, all $x\in W_y$ and all $z\in E_\ell^n$ with $\ell<k$ for which $W_z\subset W_y$. The space $X^n$ is called the $n$-th level approximation of the space $X$.
\end{definition}

We explain the notation and show that this is the natural definition. It is very similar to the geometry of polyhedra with faces, edges, and vertices.
The index set $K$ contains those $k$ for which complete $k$-tuples exist. For the square we had $K=\{ 2,4\}$, for Figures~\ref{fig4} and~\ref{fig5} we had $K=\{ 3\}$, for Example~\ref{ex3} we had $K=\{ 2,4,6,12\}$. On the level $n$, we consider the words $w$ of length $n$, the corresponding disjoint cylinders $S_w$ in the symbolic space $S$, and the pieces $X_w=\varphi (S_w)$ of the topological space $X$ which are not disjoint. The pieces $X_w$ with $w\in D^n$ form a~cover of $X$. We turn this cover into a~partition by taking all possible intersections of sets $X_w$. These sets, the atoms of the partition, are considered as points of $X^n $. They are given by the topology which is inherited from the partition sets.

For the $2\times 2$ square, the $X_w$ are $4^n$ closed squares for $n=1,2,\dots $. Two neighboring squares will intersect in an edge, and four mutually neighboring squares have one common vertex point. The partition, shown in Figure~\ref{figsquare}, will consist of vertex points, edges without vertices, and open squares. The open squares correspond to $D^n$ and are open as points in $X^n $. The minimal neighborhood $U_x$ of the edges contains the two neighboring open squares but no vertices. And the vertices are closed and their $U_x$ contains all four adjoining edges and four adjoining open squares.

For $n=1,2,\dots $. we get a~sequence of finer and finer partitions with smaller and smaller sets which in the limit coincides with the space $X$. The accurate concept is the inverse limit of the $X^n $. There is a~natural projection $\pi_{n+1,n}: X^{n+1}\to X^n$ which assigns to $(w_1,\dots,w_{n+1})\in D^{n+1}$ the word $(w_1,\dots,w_n)$ in $D^n $, and to each path $y\in E_k^{n+1}$ the path of its first $n$ edges which belongs to $E_k^n $. This projection is continuous, that is, for each open set in $X^n$ the corresponding refined set in $X^{n+1}$ is also open.

The inverse limit $X^\infty$ of $X^1 \xleftarrow{\mathrlap{\pi_{2,1}}\phantom{\pi_{2,1}}} X^2 \xleftarrow{\mathrlap{\pi_{3,2}}\phantom{\pi_{3,2}}} X^3 \dots$ is the set of all sequences 
\[ x=(x^1,x^2,\dots) \text{ with } x^n \in X^n \text{ and } \pi_{n+1,n} (x^{n+1})=x^n. \] In our case, the $x^n \in D^n$ extend to sequences $x\in D^\infty =S$, and the edge paths $x^n \in E_k^n$ extend to infinite edge paths $x$. Let $E_k^\infty$ denote the set of infinite edge paths in $G_k$ starting in $o$, and
\[
E^\infty=\bigcup_{k\in K} E_k^\infty \quad \mbox{ and }\quad X^\infty =D^\infty\cup E^\infty .
\]
The projection $\pi_n:X^\infty\to X^n$ assigns to each infinite sequence or path the finite part of the first $n$ elements. The topology on $X^\infty$ is obtained by saying that for $x\in X^\infty, $ a~base of open neighborhoods is given by $U_n(x)=\pi_n^{-1}(U(x^n))$ where $U(x^n)$ is the minimal neighhborhood of $x^n$ in $X^n $. For $x\in D^\infty=S$ these are the basic cylinder sets in the product topology. For $x\in E^\infty$ it is just the $n$-th approximation of $x$ together with all its neighbors in $E^n $.

\begin{theorem}[Quotient space as inverse limit]\label{theo3} 
Let $G$ be a~topology-generating automaton with finite equivalence classes, with size smaller than a~constant $C$. Let $\varphi: S\to X$ denote the address map, and $X^n$ the finite spaces constructed above. Then $X$ is the compact Hausdorff space associated with the inverse limit $X^\infty$.
\end{theorem}

 \begin{proof} This statement mainly says that our construction of the approximations $X^n$ is mathematically correct. We explain how the elements of $E^\infty$ will disappear in the limit. A finite topological space can never be Hausdorff unless all sets are open. However, for an infinite space one usually requires that the intersection of all neighborhoods of some point $x$ is the point itself. In $X^\infty$ this holds for the points of $D^\infty$ but not for $x\in E^\infty $. Such an $x$ represents an infinite edge path with $k$ address sequences $s_1,\dots,s_k\in S\subset X^\infty $. Each basic neighborhood of $x$ includes some finite part of these sequences in $X^n$ and thus includes
$s_1,\dots,s_k$ in $X^\infty $. In order to obtain a~Hausdorff space, we must identify $x$ with $s_1,\dots,s_k$. One way to define this identification is to say that $x\equiv y$ if $f(x)=f(y)$ for every continuous function on the space. This identification gets rid of $x$ and at the same time identifies the points $s_1,\dots,s_k\in S$. Thus we get exactly the quotient space $X$ as associated compact Hausdorff space of $X^\infty $. To prove the homeomorphism, we state that all basic neighborhoods in $X^\infty$ obviously correspond to open sets in $X$. On the other hand, given a~neighborhood $V$ of $x\in X$ as a~finite union of cylinder sets in $S$, it is easy to determine an $n$ so that the minimal neighborhood $U(x^n)$ leads to a~subset of $V$. \hfill \end{proof}

Formally, the theorem remains true even if we would define the approximations $X^n$ only with the graph $G_2$, without multiple addresses. In that case the identification on $X^\infty$ would work because of the transitivity of the relation $\varphi(x)=\varphi(y)$. On the finite stage, however, the identification could not be anticipated. For the square, for instance, we would have a~hole in the center of $X^1$ and $(4^n -1)/3$ holes in $X^n$ but no holes in the limit.

\section{Topological properties}\label{prope}
The approximations $X^n$ are meaningful. The interval is characterized already by $X^2 $. Any open point is replaced by two open points and a~closed centre point on the next level. Figure~\ref{figsquare} shows a~similar situation for the square. Each open point is replaced by $X^1$ in the next level, and each line segment splits into two line segments with a~middle point. The closed points do not split further.

It seems that topological properties of $X$ can be determined from $X^n$ on certain finite level $n$. When we get into the cyclic part of $G$, the changes of the $X^n$ will repeat those changes which were made for $X^k$ with $k<n$. The structure of the approximations will not further change, it will only be refined. However, the critical $n$ depends on the topological property as well as on the automaton and can be difficult to determine. We briefly comment on some known results, open problems and work in progress.

\paragraph*{Connectedness. }
A classical result of Hata~\cite{Hata} says that $X$ is connected if and only if $X^1$ is connected. In the case of two digits, $X$ is connected if there is at least one double address~\cite{Bar}. A connected automata-generated space is locally connected and path connected, cf.~\cite[Section 3]{BK}. There seem to be no results on $k$-connectedness. A space is $k$-connected if for any two points $x,y$ there are $k$ connecting paths which are disjoint, except for the endpoints.

\paragraph*{Disconnectedness. }
For more than two digits, there is no algorithm to decide whether $X$ is totally disconnected. Even for particular self-similar sets in space, or even in the plane, this is hard to decide~\cite{EFG1}. Luo and Xiong~\cite{Luo21} have recently shown that $X$ is totally disconnected if and only if the cardinality of connected subsets in the $X^n$ is uniformly bounded by a~constant. It remains to determine the critical $n$. If $X$ is neither connected nor totally disconnected, there is the question to describe the connected components of $X$. Can this be done by automata? A special case was studied in~\cite{xiao21}.

\paragraph*{Cutpoints. } For a~connected $X$, a~point $x$ is a~cutpoint if $X\setminus \{ x\}$ is disconnected. If $x$ has a~connected neighborhood $U$, it is a~local cutpoint if $U \setminus \{ x\}$ is disconnected. Isolated loops in the automaton $G$ lead to local cutpoints, see Proposition~\ref{pcf}. However, there are other cutpoints of a~global nature, for instance the point with address $\overline{1}$ in Figure~\ref{fig5}. The critical level for their detection could be $n=3$. For tiles, the question was studied~\cite{LL07,ALT} and for fractal squares in~\cite{RWX22}.

\paragraph*{Topological dimension and embedding dimension. }
There is a~theory of topological dimension which can be applied also to finite spaces. Is it possible to determine the dimension of $X$ on finite level $n$? We can also ask for the smallest $n$ for which $X$ is embeddable in $\RR^n $. For graphs $H$, there is a~classical theorem of Kuratowski which says that $H$ is embeddable in the plane if and only if it has neither the complete graph $K_5$ nor the complete bipartite graph $K_{3,3}$ as homeomorphic subset. This fits well with finite spaces. We apply the simple part of this theorem to Example~\ref{ex2}.

\begin{proposition}
 
Suppose some $X^n$ contains a~discrete version of $K_5$ or $K_{3,3}$, where the edges are closed discrete arcs containing at least one point of $D^n $, and the vertices are projections of connected sets in $X$. Then $X$ is not embeddable in the plane. In particular, Example~\ref{ex2} is not embeddable in the plane.
\label{plan}
\end{proposition}

 \begin{proof} An ordered set $x_1,\dots,x_n$ is a~discrete arc if for $k=1,\dots,n-1$ either $x_k$ is contained in the minimal neighborhood $U_{x_{k+1}}$ or $x_{k+1}$ is in $U_{x_k}$. These are the images of $[0,1]$ under continuous mappings, and path connectedness in finite spaces is the same as connectedness~\cite{Bam}. If a~graph $K$ is contained with discrete arcs in $X^n $, we of course require that different arcs do not meet except at their endpoints. We made strong assumptions concerning the lift of $K$ to $X$ by the projection $\pi:X\to X^n $. For each edge $e$ of $K$, the set $\pi^{-1}(e)$ is a~connected compact subset of $X$ which contains a~piece $X_w$ of $X$ since $e$ contains a~point of $D^n $. For each vertex $c$ of $K$, the set $\pi^{-1}(c)$ is connected. We can choose a~point from each $\pi^{-1}(c)$ and connect them by arcs in each $\pi^{-1}(e)$ to get a~subset of $X$ which topologically realizes the graph $K$ with proper arcs. Thus $X$ cannot be a~subset of the plane.

In Example~\ref{ex2}, the equivalence relation is given only by triple addresses, indicated as triangles in Figure~\ref{fig5}. Three pieces of $X$ corresponding to words $v0w, v1w, v2w$ will intersect, according to Figure~\ref{figthree}. Here $v$ can be any word, and $w$ can be empty, $1$, or $1u$ with $u\in\{ 0,2\}^* $. On level $n$, the pieces are represented by open points in $D^n $, and their intersection is represented as a~closed point in $X^n$ which we denote $vYw$ since it is in the closure of $v0w$, of $v1w$, and of $v2w$. Note that the closed points do not further split in $X^m, m>n$. They directly correspond to points in $X$.

We construct $K_{3,3}$ in $X^4 $. Let $012Y, 112Y, 212Y$ be the red vertices (with empty $w$) and $Y120, Y122$ and $110Y$ the blue vertices. We have to define an arc from each red vertex to each blue vertex. For the first two blue vertices, this arc consists of just one open point between the two endpoints. Namely, $012Y$ connects to $Y120$ via $0120$, and to $Y122$ via $0122$. Similar for $112Y$ and $212Y$. The connections to $110Y$ are a~bit longer:
\[
012Y-0121-01Y1-0101-010Y-0102-Y102-1102-110Y,
\]
\[
112Y-1121-11Y1-1101-110Y \quad \mbox{ and }
\]
\[
212Y-2121-21Y1-2101-210Y-2100-Y100-1100-110Y.
\]
Since each open point has only two adjacent closed points, and beside the vertices no point appears twice in our list, the edges are closed sets in $X^4$ which meet only at their endpoints. \hfill \end{proof}

\paragraph*{Homotopy and homology. }
In algebraic topology, finite spaces have become a~convenient tool~\cite{Bam,Fern}. On the other hand, various attempts were made to determine algebraic invariants like the Euler characteristic for fractal spaces. However, the difference from a~doughnut is that holes in fractals come in infinite families. So we must ask for characteristic exponents rather than absolute invariants. A first step would be to characterize contractible spaces which include trees~\cite{DT23} and disk-like tiles~\cite{ALT,LL07,LZ17,TZ20}.

\section{Realizing automata as neighbor graphs of IFS}\label{IFSex}
We now consider metric realizations of our spaces $X$ which already appeared in the Figures~\ref{fig3},~\ref{fig4},~\ref{b1},~\ref{b3}, and~\ref{fig7}. A finite set of one-to-one contractive maps $f_0,\dots,f_{m-1}$ on $\RR^d$ is called an iterated function system (IFS)~\cite{Bar}. There is a~unique compact set $X\subset \RR^d$ which fulfils $X=f_0(X)\cup\dots\cup f_{m-1}(X)$. This is the self-similarity equation~\eqref{sel1} with prescribed homeomorphisms.
Usually, the $f_k$ are assumed to be linear maps or even similitudes. As in~\cite{EFG1,EFG2, EFG3} we assume that the $f_k$ are similitudes with the same similarity ratio $r<1$. Compositions of the $f_k$ are expressed by words $u\in D^n$ as $f_u=f_{u_1}\cdot\dots\cdot f_{u_n}$.

In this case~\cite{EFG1,BM09,Mth}, as well as in the case of self-affine tilings~\cite{ALT,DN05,LL13,LL07,Lo,LZ17,ScT,TZ20}, a~topology-generating automaton can be determined directly from the IFS. It was termed neighbor map in~\cite{BM09} since the states are neighbor maps of the form $h=f_u^{-1}f_v$. They can be determined recursively, starting with $h=f_i^{-1}f_j$ for $i,j\in D$ and repeatedly applying the formula
\begin{equation}
h'=f_i^{-1}hf_j
\label{recur}
\end{equation}
to the previously determined maps $h$. We consider these maps as vertices of a~graph and put an edge with the label $(i,j)$ from $h$ to $h'$. The initial state is the identity map $id$, and since $f_j^{-1}f_j=id$, we have loops from $id$ to itself for every $j\in D$.

As described in Frougny and Sakarovich~\cite{FS09}, there is a~trim part of the graph, since by contractivity, maps $h$ with $|h(0)|>C$ for some constant $C$ will fulfil $|h'(0)|>|h(0)|$ in~\eqref{recur} for all $i,j\in D$. Such maps $h'$ are neglected, as well as all vertices $h$ which lead only to such maps. If the maps are defined in a~discrete setting, for instance a~canonical number system~\cite{FS09} or an algebraic number field generated by a~Pisot number~\cite{EFG3}, the trim part of the automaton will be finite, and the resulting automaton will fulfil the conditions of definition~\ref{def1}. In this case we shall say that the given IFS $f_0,\dots,f_{m-1}$ is a~\textit{linear or similitude representation of the automaton} $G$.

\textbf{Questions.} \textit{Given a~topology-generating automaton, does it have linear representa\-tions? How can we find them? Are they unique?}

We provide no clear answers, but some methodology and examples.
Let us start with the basic Example~\ref{ex1}. The state \textit{right} in Figure~\ref{fig1} is now called $h=f_0^{-1}f_1$, since the edge with label $(0,1)$ from $o=id$ leads to this state. We are looking for the unknown mappings $f_0,f_1$. The loop with label $(1,0)$ at this state now turns into an equation $h'=h$ in~\eqref{recur}:
\begin{equation}
f_1^{-1}hf_0 =h\, \mbox{ or }\ hf_0=f_1h.\label{e1}
\end{equation}
Such commutativity relations are typical conditions in our representation problem. We now change the notation. Let $g$ be an expanding linear map with factor $R>1$, and let $f_k=g^{-1}h_k$ for $k\in D$. Then $f_k^{-1}=h_k^{-1}g$, and $h_k$ is an isometry. Taking $g=f_0^{-1}$ and shifting the origin of our coordinate system we can assume that $g$ is a~linear map and $h_0=id$ (cf.~\cite[Section 1]{EFG3}). Moreover, $f_i^{-1}f_j=h_i^{-1}gg^{-1}h_j=h_i^{-1}h_j$.

In our case, $h=h_1$ and~\eqref{e1} turns into $h_1g^{-1}=g^{-1}h_1h_1$. In other words,
\begin{equation}
gh_1g^{-1}=h_1^2 .\label{e2}
\end{equation}
This is a~nice generating relation. Now let us assume $g$ is an orientation-preserving similitude in the complex plane: $g(z)=\lambda z$ with $\lambda\in \CC,\ |\lambda|=R$. For the isometry $h_1$ we can write $h_1(z)=az+b$ with $a,b \in\CC, |a|=1$. Here $b\not= 0$ since otherwise $X={0}$. Equation~\eqref{e2} now turns into
\[
\lambda (a\frac{z}{\lambda}+b)=a(az+b)+b \quad\mbox{ or } az+\lambda b=a^2 z +(a+1)b.
\]
This equation has the unique solution $a=1, \lambda=2$. Thus we arrive at the binary number system $g(z)=2z, h_1(z)=z+1$. The number $b$ is still a~degree of freedom in the choice of the unit point of the coordinate system.

As a~second example, we try the method for Figure 2a with the same assumptions on $g,h_0,h_1$. We have $b=h_1^{-1}$ and $c=h_1$. The edge from $c$ to $b$ reads $f_0^{-1}cf_1=b$, that is, $gh_1g^{-1}=h_1^{-2}$ instead of~\eqref{e2}. Inserting $\lambda, a, b$ we get $a^3 =1, \lambda=-(a^2 +a)$. Since $\lambda\not= 1$, we again get a~unique solution: $a=1, \lambda=-2$.

The case of Figure 2b is a~bit different. At state $b$, we have $h_1=h_1^{-1}$ which implies $h_1(z)=-z+1$ where the number $b$ was standardized. Now $c=f_1^{-1}h_1f_1=h_1gh_1g^{-1}h_1$ and $c(z)=-z+2-\lambda $. The loop at state $c$ reads $gcg^{-1}=c$ which leads to the quadratic equation $\lambda^2 -3\lambda+2=0$ and $\lambda=2, c(z)=-z$. For $X=[0,\frac12 ]$, the two neighbor maps $h_1$ and $c$ are the point reflections at the endpoints of the interval. We proved a~supplement to Proposition~\ref{prop3}.

\begin{proposition}
For the automata in Figures~\ref{fig1} and~\ref{fig2}, the only representations by complex linear maps are the IFS which belong to the number systems with base 2 and -2, and to the tent map. \hfill $\Box$
\label{numb}
\end{proposition}

The representation remained unique although we went from dimension one to two. However, the uniqueness is lost if we allow for orientation-reversing maps $g(z)=\lambda \overline{z}$. Besides the binary interval, we then get Koch curves generated by two reflective maps.

For Example~\ref{ex3} it will now be shown that the whole algebra of the reflection group as well as the action of $g$ is contained in the automaton in Figure~\ref{fig7}.

\begin{proposition}
For the three-state automaton in Figure~\ref{fig7}, there is a~unique orientation-preserving similitude $g$ and neighbor maps $a,b,c$ which represent this automaton. They satisfy
\[
a^2 =b^2 =c^2 =id \, \ ab=ba\, \ (ac)^3 =id\, \ (cb)^6 =id\quad \mbox{ and }
\]
\[
gcg^{-1}=b\, \ gag^{-1}=cbc \, \ gbg^{-1}=cac=aca.
\]
\label{cox}
\end{proposition}

\begin{proof} We work in an abstract setting not restricted to $\CC $.
It is crucial to set $h_1=id$ because $h_0=id$ would not lead to an orientation-preserving $g$. The edges from $o=id$ then imply $c=h_0=h_0^{-1}$ and $a=h_2=h_2^{-1}$. This shows $a^2 =c^2 =id$. Now consider the two edges from $c$ to $b$ and the basic recursion~\eqref{recur} where $(i,j)$ are the edge labels.
\[
b=f_1^{-1}cf_1=gcg^{-1}\, \ b=f_2^{-1}cf_2=h_2gcg^{-1}h_2=aba .
\]
The first relation and $c^2 =id$ yields $b^2 =gcg^{-1}gcg^{-1}=id$. The second relation gives $ab=ba$, which is also self-inverse so that we could conclude $(ab)^2 =id$. The edge from $a$ to $b$ is evaluated as
\[
b=f_0^{-1}af_0= cgag^{-1}c \, \ \mbox{ or }\ cbc=gag^{-1}.
\]
The edge from $b$ to $a$ gives
\[
a=f_0^{-1}bf_0=cgbg^{-1}c \, \ \mbox{ or }\ cac=gbg^{-1}\,
\]
and the edge from $b$ to $c$ gives $aca=gbg^{-1}$. Multiplying the last two equations we get $(ac)^3 =gb^2 g^{-1}=id$. To finish the proof, we determine
\[
cbcb=(cbc)b=gag^{-1}gcg^{-1}=g(ac)g^{-1}.
\]
Taking the third power, we obtain $(cb)^6 =id$. It is now clear that $ab,ac$, and $bc$ are rotations and $a,b,c$ must be reflections. The first line in the proposition shows the well-known Coxeter relations for the reflection group generated by $a,b,c$. The second line expresses the action of $g$ as an automorphism on that group, similar to a~substitution. They do determine $g$. If $X$ is the triangle formed by the reflection lines of $a,b,c$, then the reflection lines of $cbc, cac$, and $b$ determine another triangle which is $g(X)=a(X)\cup X \cup c(X)$. \hfill 
\end{proof}

\begin{conjecture} For any self-similar crystallographic or self-affine tile, the combinatorial structure of the neighbor automaton uniquely determines the IFS.
\end{conjecture}

The next example shows that we probably can extend this conjecture to fractals which are doubly connected, that is, there are two disjoint paths between any pair of points. In~\cite{EFG2} we studied a~connected self-similar set $X$ with Cantor sets $f_i(X)\cap f_j(X)$ for which the IFS was not unique. For disconnected sets like Figure~\ref{fig3} or trees like Figure~\ref{fig4} it is clear that there are many different representations by similitudes.

\begin{example}[The dog carpet with five-state automaton]\label{ex4}
The self-similar set $X$ in Figure~\ref{dog} is generated by an IFS with five maps and data from the algebraic number field $\QQ(\sqrt{-15})$. We call it dog carpet because the holes look like a~pet. In~\cite[Figure 10]{EFG2} we found a~symmetric version of this fractal with five IFS maps and 12 neighbor maps. This $X$ has an incomplete neighbor automaton with only five states. The curious point is that the maps involve rotations around irrational angles so that the pieces have an infinite number of orientations when we extend the construction over the plane. Examples of this type have not been considered in the literature, and no tiling of this type is expected to exist.
\end{example}

\begin{figure}[h!t]
\begin{center}
\includegraphics[width=0.45\textwidth]{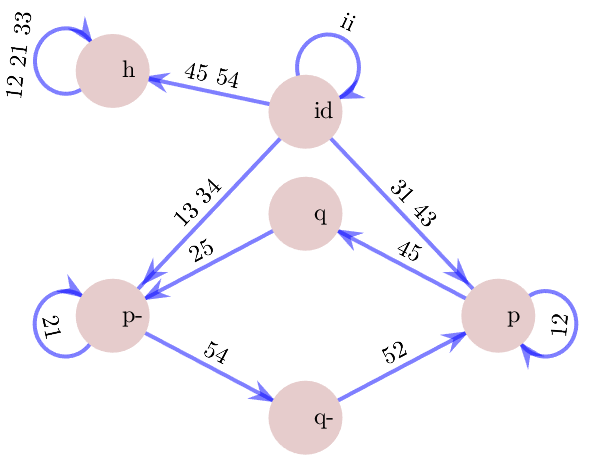} \quad
\includegraphics[width=0.45\textwidth]{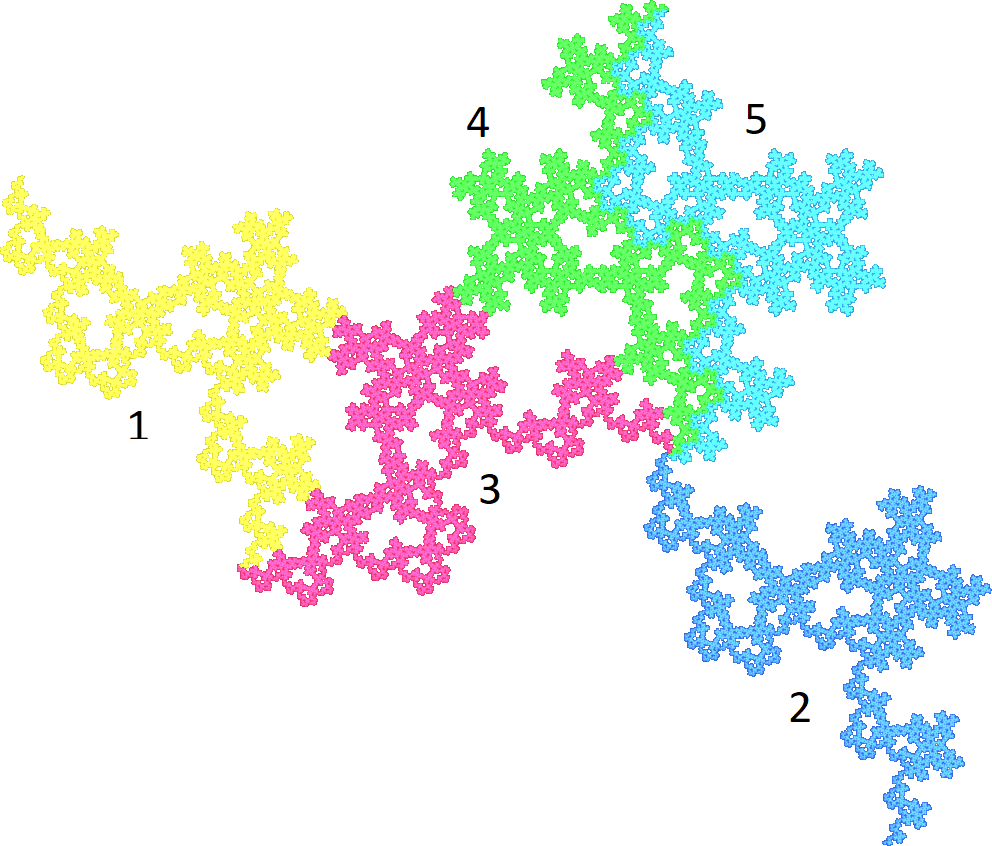} \\
\includegraphics[width=0.8\textwidth]{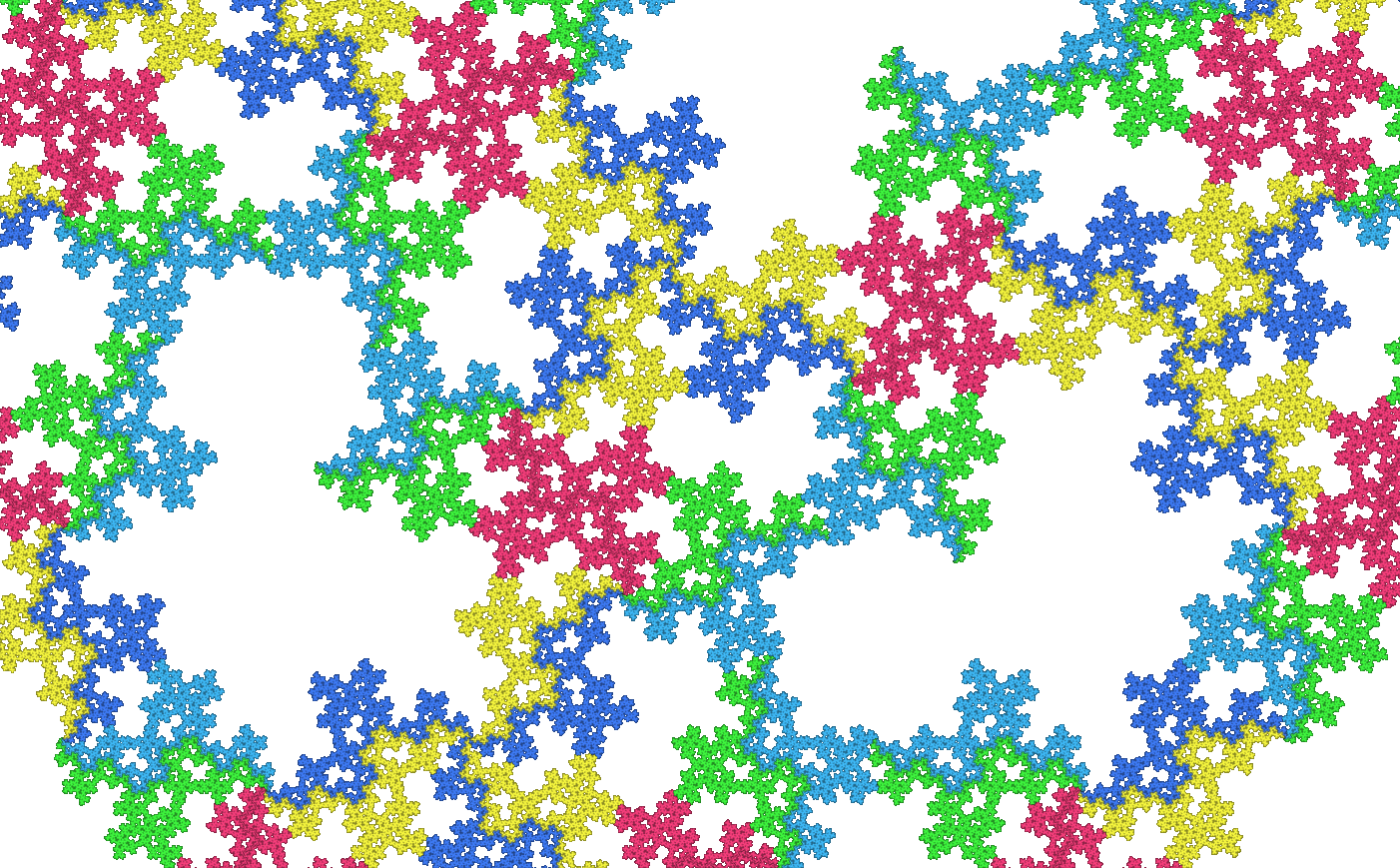}
\end{center}
\caption{Top: a~self-similar set $X$ and its neighbor automaton. Bottom: a~magnification which shows the fractal tiling structure. The IFS maps $f_k$ are determined by the automaton. This seems curious since the neighbor maps involve an irrational angle.}\label{dog}
\end{figure}

\begin{proposition}
 
The automaton in Figure~\ref{dog} has a~unique representation by orientation-preserving similitudes in the plane.
\label{uni}
\end{proposition}

 \begin{proof} Let $g(z)=\lambda z$ and $f_k=g^{-1}h_k$ with orientation-preserving isometries $h_k$ for $k=1,\dots,5$. We determine the data of this IFS from the automaton. We choose the origin and unit point so that
\[
h_3=id \ \mbox{ and }\ h_1(z)=az+1 \ \mbox{ with } \ |a|=1.
\]
On the edge from $id$ to $p$ we have $p=f_3^{-1}f_1=f_4^{-1}f_3$ which implies $h_1=h_4^{-1}$. Thus
\[
h_4=h_1^{-1}=\overline{a}(z-1) \ \mbox{ and }\ p=h_1.
\]
The edge from $id$ to $h$ has labels 45 and 54, so $h$ must be self-inverse, $h(z)=-z+v$. One loop at $h$ has label 33 which means $f_3^{-1}hf_3=h$ or $hf_3=f_3h$. Thus $h$ must have the same fixed point as $f_3=g^{-1}$, that is, the origin. Thus
\[
h(z)=-z\,, \ \mbox{ and } \ h_5=h_4h=-\overline{a}(z+1)
\]
because $h= f_4^{-1}f_5=h_4^{-1}h_5$. The loop with label 12 at state $h$ says
$f_1^{-1}hf_2=h$, or $f_2=hf_1h$. Here $g^{-1}$ cancels out and we get
\[
h_2=az-1.
\]
Only two unknowns remain: $a$ and $\lambda $. The loop with label 12 at state $p=h_1$ gives
\[
\frac{a}{\lambda}(az-1)+1 = h_1f_2=f_1h_1= \frac{1}{\lambda}(a^2 +a+1)
\]
\[
\mbox{ which results in }\ \lambda=2a+1.
\]
Finally, we consider the path from state $p=h_1$ to the inverse state $p-=h_1^{-1}$ via $q$. The corresponding equation is $f_2^{-1}f_4^{-1}h_1f_5^2 =h_1^{-1}=h_4$. Hence
\[
\frac{\overline{a}}{\lambda^2}(z+1)-\frac{1}{\lambda}+1 = h_1f_5^2 =f_4f_2(\overline{a}(z-1))= \frac{\overline{a}}{\lambda^2}(z-2)-\frac{\overline{a}}{\lambda} .
\]
The resulting equation $\lambda^2 +\lambda(\overline{a}-1)+3\overline{a}=0$ is multiplied with $a$ since $a\overline{a}=1$. Then we substitute $a=(\lambda-1)/2$. We get the cubic equation $\lambda^3 -2\lambda^2 +3\lambda+6=0$. Dividing by $\lambda+1$, we arrive at the quadratic equation
\[
\lambda^2 -3\lambda+6=0\ \mbox{ with solution } \lambda=\frac12 (3+i\sqrt{15})\, \ a=\frac14 (1+i\sqrt{15}).
\]
Note that $a$ describes the rotation between pieces $X_3$ and $X_1$ in Figure~\ref{dog}. Since $2a$ is a~quadratic integer, $a$ is not a~root of unity, and the rotation angle is irrational. \hfill \end{proof}

This proof shows once more how much information can be encoded in a~small automaton.

\section{Conclusion}
This paper is by no means complete. It is a~starting point for further investigation. Open problems were listed in Sections~\ref{prope} and~\ref{IFSex}. In Section~\ref{topsel} we pointed out that topology-generating automata should be considered in the broader context of finite type symbolic spaces and graph self-similarity. Our assumption of finite equivalence classes should be canceled and the decision on the size of equivalence classes included into the algorithm of Section~\ref{comp}. Most importantly, our two algorithms for automata of multiple addresses and finite approximation spaces must be programmed efficiently so that they can be applied to larger automata $G$. Then new spaces can be generated from automata, and complicated self-similar sets as in~\cite{EFG1,EFG3} can be properly analyzed. Today's computers provide a~chance to realize the vision of Mandelbrot~\cite{FGN} and Barnsley~\cite{Bar} expressed 50 years ago: to model natural geometric phenomena like dust, soil, foam, snow, smoke, clouds. Perhaps automata can also help classify tilings and fractal spaces. On the theoretical side, the isomorphism problems indicated at the end of Section~\ref{topsel} seem most important.

\subsection*{Acknowledgments}
This paper is dedicated to Christiane Frougny. Without her, it would not have been written. Many years ago, she invited me to Paris and introduced me to her numeration community which profoundly influenced my direction of research. I am still very grateful to Christiane.

\EditInfo{December 5, 2023}{April 25, 2024}{Rigo Michel, Emilie Charlier and Julien Leroy}
\end{document}